\documentclass[review]{elsarticle}
\usepackage{}

\usepackage[latin1]{inputenc} 
\usepackage{listings, tcolorbox}

\usepackage{lineno,hyperref}
\usepackage{amssymb}
\usepackage{color}
\usepackage{amsmath}
\usepackage{tikz}
\usepackage{amsfonts}
\usepackage{mathrsfs}
\usepackage{amsthm}

\usepackage{bbding}
\usepackage{mathrsfs}
\usepackage{amsmath, amsfonts, amssymb, mathrsfs, txfonts}
\usepackage{graphicx, subfigure}
\usepackage{wasysym}
\usepackage{color}
\usepackage{bm}
\usepackage{enumerate}

\usepackage{pifont}
\usepackage{txfonts}
\usepackage{amsmath}
\usepackage{graphicx}
\usepackage{amsfonts}
\usepackage{amssymb}
\usepackage{mathrsfs,psfrag,eepic,epsfig}
\usepackage{epstopdf}

\modulolinenumbers[5]

\journal{Journal of \LaTeX\ Templates}

\makeatletter \@addtoreset{equation}{section}

\newtheorem{thm}{Theorem}[section]
\newtheorem{cor}[thm]{Corollary}

\newtheorem{prop}[thm]{Proposition}
\theoremstyle{definition}

\newtheorem{rem}[thm]{Remark}
\newtheorem{assum}[thm]{Assumption}
\newtheorem{RHP}[thm]{RH Problem}

\renewcommand{\baselinestretch}{1.25}
\begin{document}

\begin{frontmatter}

\title{Soliton resolution for the Hirota equation with weighted Sobolev initial data \tnoteref{mytitlenote}}
\tnotetext[mytitlenote]{
Corresponding author.\\
\hspace*{3ex}\emph{E-mail addresses}: sftian@cumt.edu.cn,
shoufu2006@126.com (S. F. Tian) }

\author{Jin-Jie Yang, Shou-Fu Tian$^{*}$ and Zhi-Qiang Li}
\address{School of Mathematics, China University of Mining and Technology,  Xuzhou 221116, People's Republic of China}

\begin{abstract}
In this work, the $\overline{\partial}$   steepest descent method is employed to investigate the soliton resolution for the Hirota equation with the initial value belong to weighted Sobolev space $H^{1,1}(\mathbb{R})=\{f\in L^{2}(\mathbb{R}): f',xf\in L^{2}(\mathbb{R})\}$. The long-time asymptotic behavior of the solution $q(x,t)$ is derived in any fixed space-time cone $C(x_{1},x_{2},v_{1},v_{2})=\left\{(x,t)\in \mathbb{R}\times\mathbb{R}: x=x_{0}+vt ~\text{with}~ x_{0}\in[x_{1},x_{2}]\right\}$. We show that  solution resolution  conjecture of the Hirota equation is characterized by the leading order term $\mathcal {O}(t^{-1/2})$ in the continuous spectrum, $\mathcal {N}(\mathcal {I})$ soliton solutions in the discrete spectrum and error order $\mathcal {O}(t^{-3/4})$  from the $\overline{\partial}$ equation.
\end{abstract}


\end{frontmatter}

\tableofcontents

\section{Introduction}
In 1974, employing inverse scattering transformation  to study the long-time asymptotic solutions of nonlinear equations was first proposed by Manakov \cite{Manakov-1974}. Subsequently, Zakharov and Manakov obtained the long-time asymptotic solution of the nonlinear Schr\"{o}dinger (NLS) equation with  decaying initial value  for the first time \cite{Zakharov-1976}. This method is used to study the long-time asymptotic solutions of nonlinear integrable systems, including the Landau-Lifshitz equation \cite{Bikbaev-1988} and the Maxwell-Bloch equation \cite{Bikbaev-1992}, etc. Until 1993, Deift and Zhou proposed the nonlinear steepest method rigorously obtained the long-time asymptotic behavior of the matrix modified Korteweg-de Vries (mKdV) with  decaying initial value \cite{DZ-AM}, which is called Deift-Zhou (DZ) method. Since then, this method has been gradually applied to other integrable models with many excellent results including focusing NLS \cite{DZ-1994},   short pulse equation \cite{Xu-2018,Xu-2020}, Fokas-Lenells equation \cite{Xu-2015}, Camassa-Holm equation \cite{Monvel-2009}, derivative nonlinear Schr\"{o}dinger (DNLS) equation \cite{tian-PA2018}, and other equations \cite{HG-JDE}-\cite{WX-JGP}.

Recently,  McLaughlin and Miller \cite{McLaughlin-1,McLaughlin-2} proposed a new method to analyze the oscillatory Riemann-Hilbert (RH) problem based on the DZ method and the $\overline{\partial}$ dressing method \cite{dbar-1}-\cite{dbar-5}, namely the $\overline{\partial}$ steepest descent method. This method is different from the DZ method in the following aspects:
(i) For DZ method, it is necessary to extend the jump matrix analytically, while for the $\overline{\partial}$ descent method, it is only necessary to extend continuously.
(ii) The $\overline{\partial}$ method avoids the delicate estimation involving the $L^{p}$ estimation of Cauchy projection operator.
(iii) The $\overline{\partial}$ method improves the error estimation without adding additional conditions to the initial data.
 
For finite mass initial data \cite{Dieng-2008}, this method was successfully applied to the de-focusing NLS equation. In addition, this method was used by Jenkins et.al to study the soliton decomposition conjecture of the focusing NLS equation with the initial value belongs to a weighted Sobolev space \cite{Jenkins-2018}, which decomposes the solution into the sum of a finite number of separated solitons and a radiative parts as $t\rightarrow\infty$. Furthermore, Cuccagna and  Jenkins proved the  asymptotic stability of $N$-soliton solutions of the de-focusing NLS equation by using the $\overline{\partial}$ steepest decent method \cite{Cuccagna-2016}. Of course, in recent years, some excellent results have been obtained by using this method, which not only further improves the estimation accuracy of the DZ method, but also verifies the soliton decomposition conjecture, including DNLS \cite{Jenkins-2018,Jenkins-CPDE}, short pulse equation \cite{Faneg-1},  Fokas-Lenells equation \cite{Faneg-2}, Kundu-Eckhaus equation \cite{Faneg-3}, mKdV equation \cite{JQ-19} etc.

In this work, we apply the $\overline{\partial}$ steepest decent method  to investigate the Hirota equation \cite{Hirota-1973} with  the initial value $q_{0}(x)$ belongs to the weighted Sobolev space $H^{1,1}(\mathbb{R})$
\begin{align}\label{1.1}
&iq_{t}+\alpha\left(q_{xx}+2|q|^{2}q\right)+i\beta\left(q_{xxx}+6|q|^{2}q_{x}\right)=0,\quad
\alpha,\beta\in\mathbb{R},\\
&q(x,0)=q_{0}(x)\in H^{1,1}(\mathbb{R}),
\end{align}
where $\alpha$, $\beta$ denote the second-order and third-order dispersions, and $H^{1,1}(\mathbb{R})$ is the weighted Sobolev space, i.e.,
\begin{align}\label{Sobolev-space}
 H^{1,1}(\mathbb{R})=\left\{f\in L^{2}(\mathbb{R}): f',xf\in L^{2}(\mathbb{R})\right\}.
\end{align}
The Hirota equation can be regarded as modified Schr\"{o}dinger equation with high-order dispersion and  time-delay corrections to the cubic nonlinearity. It is a more accurate approximation than the NLS equation in describing wave propagation in the ocean and optical fiber \cite{Hirota-1973}. Many properties and different kinds of exact solutions of the Hirota equation have been obtained, including rouge wave via Darboux transformation \cite{He-2012}, solutions under non-zero boundary conditions via RH probelm \cite{zhang-2020}, other aspects of research \cite{Lakshmanan-1983,Lamb-1980}. In addition, the conservation laws have been investigated in detail in \cite{Hao-2017}. From Lax pair representation, the spectrum problem is transformed into Riccati equation
\begin{align*}
W_{x}=q_{x}\frac{W}{q}-|q|^{2}+2izW-W^{2},
\end{align*}
by introducing a transformation $W=q\Gamma$ with $\Gamma= \psi_{2} \psi_{1}^{-1}$, where $\psi_{j}$ ($j=1,2$) are the column of $\psi$.  The following formula can be obtained from the compatibility condition
\begin{align}\label{law}
(-iz+W)_{t}=(A+B\Gamma)_{x},
\end{align}
where the $A$ and $B$ denote the elements $V_{11}$ and $V_{12}$ of matrix $V$ defined by \eqref{lax2}, respectively, and $W$ such that $W=\sum_{n=1}^{\infty}\varpi_{n}z^{-n}$. After calculation, it is known that $\varpi_{n}$ satisfies the following recurrence relation
\begin{align*}
-2i\varpi_{j+1}+\sum_{k=1}^{j}\varpi_{k}\varpi_{j-k}-\frac{\varpi_{j}q_{x}}{q}+\varpi_{j,x}=0,
\quad j=2,3\cdots,
\end{align*}
and then the first three conserved quantities can be expressed as
\begin{align*}
I_{1}&=-\frac{1}{2}\int_{-\infty}^{+\infty}iq(x,t)q^{*}(x,t)dx,\\
I_{2}&=-\frac{1}{4}\int_{-\infty}^{+\infty}q(x,t)q_{x}^{*}(x,t)dx,\\
I_{3}&=\frac{1}{8}\int_{-\infty}^{+\infty}iq(x,t)\left[q(x,t)(q^{*}(x,t))^{2}+q_{xx}^{*}(x,t)
\right]dx,\ldots.
\end{align*}

Beside that, Hamiltonian function can also be calculated by Lax representation \cite{Peng-2020}
\begin{align}
\frac{d\psi_{1}}{dx}=\frac{\partial H}{\partial\psi_{2}},\quad
\frac{d\psi_{1}}{dx}=-\frac{\partial H}{\partial\psi_{1}},
\end{align}
with the real valued Hamiltonian function
\begin{align}
H=z\psi_{1}\psi_{2}+\overline{z}\overline{\psi}_{1}\overline{\psi}_{2}+\frac{1}{2}
(\psi_{1}^{2}+\overline{\psi}_{2}^{2})(\psi_{2}^{2}+\overline{\psi}_{1}^{2}),
\end{align}
where $\psi=(\psi_{1}~\psi_{2})$ is a nonzero solution of the Lax pair, and the $\overline{\psi}$ denotes complex conjugation of $\psi$.

It is worth noting that the DZ method was used to study the long-time asymptotic solution of Hirota equation with the error estimate $\epsilon(x,t)=\mathcal {O}(t^{-1}\log t)$ in \cite{Huang-NA}. The initial boundary value problem of \eqref{1.1} was studied in \cite{GUO-2018}. In addition, the Hirota equation can be regarded as a mixture of nonlinear Schr\"{o}dinger equation and complex mKdV equation, because when the parameters are $\alpha=1,~ \beta=0$ and $\alpha=0,~ \beta=1$ respectively, the equation \eqref{1.1} degenerates to nonlinear Schr\"{o}dinger equation and complex mKdV equation, namely
\begin{itemize}
\item Taking $\alpha=1,~ \beta=0$, equation \eqref{1.1} reduces to the NLS equation
\begin{align}\label{1.2}
iq_{t}+  q_{xx}+2|q|^{2}q =0,
\end{align}
\item Taking $\alpha=0,~ \beta=1$, equation \eqref{1.1} reduces to the mKdV equation
\begin{align}\label{1.3}
iq_{t} + q_{xxx}+6|q|^{2}q_{x} =0.
\end{align}
                  \end{itemize}

For NLS equation \eqref{1.2}, the long-time asymptotic solution was investigated by Deift and Zhou via DZ method in  \cite{DZ-1994,DZ-CMP1994}, which show that the error term $\varepsilon(x,t)=\mathcal {O}(t^{-1}\log t)$ as the initial value $q_{0}(x)$ has a high degree of smoothness and decay. In \cite{AIHP}, Borghese and Jenkins et.al studied the initial value in the weighted Sobolev space by using the $\overline{\partial}$ steepest decent method, the error estimate is $\varepsilon(x,t)=\mathcal {O}(t^{-3/4})$.  In addition, the soliton resolution conjecture was proved. On the other hand, Tian employed Fokas method to study the initial boundary value problem of coupled NLS equations \cite{Tian-PAMS,Tian-JDE}. For mKdV equation, similar results were obtained in \cite{JQ-19}. The soliton  resolution conjecture about NLS equation and mKdV equation has been proved. Therefore in this work, we will prove that the soliton resolution conjecture about Hirota equation is also true based on the relationship between the two equations and Hirota equation.

A brief description of the scattering data for the Hirota equation \eqref{1.1} is given as follow: generally, the points where the diagonal element $s_{11}(z)$ of the scattering matrix are zero are called spectral points. When these spectral points satisfy  $s_{11}(z_{n})=0,~s_{11}'(z_{n})\neq0$, which are simple spectral points. It is worth noting that these spectral points are distributed in the whole complex $z$-plane. When $s_{11}(z)=0,~z\in R$ holds, these points are called spectral singularities. If there are spectral singularities, the spectral points may be infinite. When there are no spectral singularities, the spectral points are finite, which consist of the  finite discrete spectrum set. $c_{k}$ is called a norming constant if it is a nonzero constant associated with a simple discrete spectrum. For the initial data $q_{0}(x)$ with only simple zeros and no spectral singularities, the minimum scattering data set of Hirota equation is $D=\left\{\gamma(z),\{(z_{k},c_{k})\}_{k=1}^{N}\right\}$. We will not discuss the case of higher order zeros and spectral singularities, which is described in \cite{DZ-1989}.

In this work, we consider the long-time asymptotic solution of Hirota equation by using the $\overline{\partial}$ steepest decent method for initial value $q_{0}(x)$ with simple zeros and spectral singularities. However, it should be noted that the steady-state phase points $z_{0}$ and $z_{1}$ are asymmetric. We use the similar method in \cite{JQ-19}, that is, adding a jump contour, and we can prove that the new jump contour has no effect on the solution.\\

  \textbf{The main result and remark is as follows.}\\
We will give the main theorem of this work, the soliton resolution conjecture, that is, the solution is decomposed into the sum of a finite number of separated solitons and a radiative parts as $t\rightarrow\infty$.

\begin{thm}\label{thm-4}
Let $q(x,t)$ be the solution of equation \eqref{1.1} whose initial value $q(x,0)=q_{0}(x)$  belongs to the weighted Sobolev space $H^{1,1}(\mathbb{R})$ and satisfies \textbf{Assumption} \ref{ass1}. Denoting $D=\left\{\gamma(z),\{(z_{k},c_{k})\}_{k=1}^{N}\right\}$ is the scattering data corresponding to the initial value without spectral singularities. Fixed $x_{1},x_{2},v_{1},v_{2}\in \mathbb{R}$ with $x_{1}<x_{2}$, $v_{1}<v_{2}$,
\begin{align}
\mathcal {I}&=\left\{z:f(v_{2})<|z|<f(v_{1})\right\},\\
\mathcal {K}(\mathcal {I})&=\{z_{j}\in\mathcal {K}:z_{j}\in\mathcal {I}\},\quad
\mathcal {N}(\mathcal {I})=|\mathcal {K}(\mathcal {I})|,
\end{align}
and a cone  shown in Fig.5.
Let $m_{sol}(x,t|D(\mathcal {I}))$ be the $\mathcal {N}(\mathcal {I})$ soliton solution corresponding to the scattering data $\{z_{j},c_{j}(\mathcal {I})\}_{j=1}^{N}$ defined by \eqref{E1}, it follows from \eqref{3.16}, \eqref{7.16}, \eqref{7.35}, \eqref{9.9} and  \eqref{10.17} that
\begin{align}
m(x,t)=m_{sol}(x,t|D(\mathcal {I}))+t^{-1/2}f(x,t)+\mathcal {O}(t^{-3/4}),
\end{align}
 which implies that the long-time asymptotic behaviour can be expressed by
 \begin{align}
 q(x,t)=m_{sol}(x,t|D(\mathcal {I}))+t^{-1/2}f(x,t)+\mathcal {O}(t^{-3/4}),
\end{align}
where $f(x,t)$ is defined by \eqref{9.13}.
\end{thm}

\begin{rem}
Here we point out that the result of \textbf{Theorem} \ref{thm-4} degenerates to the work of Borghese and Jenkins \cite{AIHP} when the parameters $\alpha=1,~\beta=0$, but it is worth noting that when the parameters $\alpha=0,~\beta=1$, it is different from the work of  Chen and Liu, whose work is real mKdV in \cite{JQ-19}. In addition, dealing with the real mKdV , they gave different solutions according to different regions.
\end{rem}

\noindent \textbf{Organization of the work}

In Section 2, we construct the analytic, asymptotic and symmetry of the eigenfunctions and scattering matrix to establish a suitable RH problem for $M(z)$, from which we establish the relationship between the solution of Hirota equation and the solution of RH problem. In Section 3,  a new matrix valued function $M^{(1)}(z)$ is defined by introducing $T(z)$, which has regular discrete spectrum and two triangular decompositions at steady-state phase points. In Section 4, by introducing $R^{(2)}(z)$, we get a mixed $\overline{\partial}$-RH problem about $M^{(2)}(z)$, whose jump path changes from the real axis $\mathbb{R}$ to the the descent line. In Section 5, we decompose the mixed RH problem $M^{(2)}(z)$ into a pure RH problem and a pure $\overline{\partial}$ problem corresponding to $M_{RHP}^{(2)}(z)$ and $M^{(3)}(z)$, respectively. In section 6, the $M_{RHP}^{(2)}(z)$ can be solved by an out model $M^{(out)}(z)$ for the soliton components under the reflection-less condition. In section 7, the internal models $M^{(z_{0})}(z)$ and $M^{(z_{1})}(z)$ near the steady-state phase points $z_{0}$ and $z_{1}$ can be approximated by the known solvable parabolic cylinder model. In Section 8, the error function $E(z)$ is considered by using the small norm RH problem. In Section 9, with the help of known results, the pure $\overline{\partial}$ problem $M^{(3)}(z)$ is considered.
Finally, the soliton resolution and long-time asymptotic solution of Hirota equation are obtained via the inverse transformation
\begin{align*}
M(z)\leftrightarrows M^{(1)}(z)\leftrightarrows M^{(2)}(z)\leftrightarrows M^{(3)}(z) \leftrightarrows E(z).
\end{align*}

\section{The spectral analysis}
In this section, we mainly establish the eigenfunctions and scattering matirx related to the initial value $q(x,0)\in H^{1,1}$, and further obtain the corresponding analysis, asymmetry and symmetry. Note that the Hirota equation admits the Lax pair
\begin{align}
&\psi_{x}(x,t;z)+iz[\sigma_{3},\psi(x,t;z)]=Q\psi(x,t;z),\label{lax1}\\
&\psi_{t}(x,t;z)+\left(4i\beta z^{3}+2i\alpha z^{2}\right)[\sigma_{3},\psi(x,t;z)]=V\psi(x,t;z),\label{lax2}
\end{align}
 where
\begin{align*}
&Q=-i[\sigma_{3},\langle\psi R\rangle]=\left(
                                                \begin{array}{cc}
                                                  0 & q \\
                                                  -q^{*} & 0 \\
                                                \end{array}
                                              \right)
, \\  V=\alpha\left[2zQ+i\sigma_{3}(Q_{x}-Q^{2})\right]&+\beta\left[4z^{2}Q+
2iz\sigma_{3}(Q_{x}-Q^{2})+[Q_{x},Q]-Q_{xx}+2Q^{3}\right].
\end{align*}
Considering the asymptotic spectrum problem as $x\rightarrow\infty$, we introduce the two eigenfunctions
\begin{align}\label{3.1}
\phi_{\pm}=\psi_{\pm}e^{i\left[zx+(4\beta z^{3}+2\alpha z^{2})t\right]\sigma_{3}}\rightarrow\mathbb{I}, ~\text{as}~ x\rightarrow\infty.
\end{align}
The linear spectrum problem  can be written as
\begin{align}\label{3.2a}
\begin{split}
\phi_{x}+iz[\sigma_{3},\phi]=Q\phi,\\
\phi_{t}+(4i\beta z^{3}+2i\alpha z^{2})[\sigma_{3},\phi]=V\phi,
\end{split}
\end{align}
which can be written as a fully differential form
\begin{align}\label{3.3}
d\left(e^{i\left[zx+(4\beta z^{3}+2\alpha z^{2})t\right]\widehat{\sigma}_{3}}\phi\right)=
e^{i\left[zx+(4\beta z^{3}+2\alpha z^{2})t\right]\widehat{\sigma}_{3}}\left(Qdx+Vdt\right)\phi.
\end{align}
From Eq.\eqref{3.3}, we know that the integral has nothing to do with the path, so we choose two special paths $(x_{1},t_{1})=(-\infty,t)$ and $(x_{2},t_{2})=(+\infty,t)$ shown in Fig. 1. Then $\phi_{\pm}$ are determined by the following integral equations
\begin{align}\label{3.4}
\begin{split}
&\phi_{-}(x;z)=\mathbb{I}+\int_{-\infty}^{x}e^{-iz(x-y)\widehat{\sigma}_{3}}Q(y;z)\phi_{-}(y;z)dy,\\
&\phi_{+}(x;z)=\mathbb{I}-\int_{x}^{+\infty}e^{-iz(x-y)\widehat{\sigma}_{3}}Q(y;z)\phi_{+}(y;z)dy.
\end{split}
\end{align}
\centerline{\begin{tikzpicture}[scale=0.4]
\draw[-][thick](-8,7)--(-2,7)[->][thick];
\draw[-][thick](8,7)--(2,7)[->][thick];
\draw[-][thick](-9,5)--(-8,5);
\draw[-][thick](-8,5)--(-7,5);
\draw[-][thick](-7,5)--(-6,5);
\draw[-][thick](-6,5)--(-5,5);
\draw[-][thick](-5,5)--(-4,5);
\draw[-][thick](-4,5)--(-3,5);
\draw[-][thick](-3,5)--(-2,5);
\draw[-][thick](-2,5)--(-1,5)[->][thick]node[below]{$x$};
\draw[-][thick](-5,1)--(-5,2);
\draw[-][thick](-5,2)--(-5,3);
\draw[-][thick](-5,3)--(-5,4);
\draw[-][thick](-5,4)--(-5,5);
\draw[-][thick](-5,5)--(-5,6);
\draw[-][thick](-5,6)--(-5,7);
\draw[-][thick](-5,7)--(-5,8);
\draw[-][thick](-5,8)--(-5,9)[->] [thick]node[right]{$t$};
\draw[-][thick](1,5)--(2,5);
\draw[-][thick](2,5)--(3,5);
\draw[-][thick](3,5)--(4,5);
\draw[-][thick](4,5)--(5,5);
\draw[-][thick](5,5)--(6,5);
\draw[-][thick](6,5)--(7,5);
\draw[-][thick](7,5)--(8,5);
\draw[-][thick](8,5)--(9,5)[->][thick]node[below]{$x$};
\draw[-][thick](5,1)--(5,2);
\draw[-][thick](5,2)--(5,3);
\draw[-][thick](5,3)--(5,4);
\draw[-][thick](5,4)--(5,5);
\draw[-][thick](5,5)--(5,6);
\draw[-][thick](5,6)--(5,7);
\draw[-][thick](5,7)--(5,8);
\draw[-][thick](5,8)--(5,9)[->] [thick]node[right]{$t$};
\end{tikzpicture}}
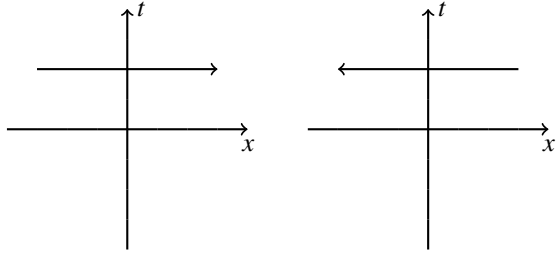
\centerline{\noindent {\small \textbf{Figure 1.} The contour of integral \eqref{3.4}.}}

From the integral equations \eqref{3.4}, the following results can be proved similar to those in \cite{BC-1984,BC-1988}.
\begin{prop}\label{prop2}
Assume that the initial value $q(x)-q_{0}\in H^{1,1}$, the eigenfunctions $\phi_{\pm}$ admit\\
$\bullet$ $\phi_{-,1}$ and $\phi_{+,2}$ are analytic in $\mathbb{C}^{+}=\left\{z|Imz>0\right\}$, and continuous to the real axis $\mathbb{R}$,\\
$\bullet$ $\phi_{-,2}$ and $\phi_{+,1}$ are analytic in $\mathbb{C}^{-}=\left\{z|Imz<0\right\}$, and continuous to the real axis $\mathbb{R}$,
where $\phi_{\pm,j}$ $(j=1,2)$ denote the $j-th$ column of $\phi_{\pm}$.
\end{prop}

For the first-order linear spectral problem, since both $\phi_{-}$ and $\phi_{+}$ are solutions, there exists a matrix $S(z)=(s_{ij})_{2\times2}$ which is independent of time and space such that
\begin{align}\label{3.5}
\phi_{-}(x;z)=\phi_{+}(x;z)S(z),\quad x,t,z\in R.
\end{align}
Note that tr$Q=0$, it follows from the Liouville's formula \cite{Liu} that $(\det \phi)_{x}=(\det \phi)_{t}=0$, furthermore, $\det \psi_{\pm}=\det \phi_{\pm}=1$ can be derived by Eq.\eqref{3.1}. The elements of the scattering matrix $S(z)$ can be expressed as eigenfunctions by direct calculation as follows
\begin{align}\label{3.6}
s_{11}(z)=Wr\left(\phi_{-,1},\phi_{+,2}\right),\quad
s_{22}(z)=Wr\left(\phi_{+,1},\phi_{-,2}\right),
\end{align}
where $Wr$ represents the Wronskians determinant. Combining with the Proposition \eqref{prop2} and Eq.\eqref{3.6},  we have that the scattering data $s_{11}$ is analytic in $\mathbb{C}^{+}$, and $s_{22}$ is analytic in $\mathbb{C}^{-}$.

Next, we consider the symmetry of the eigenfunctions and the scattering matrix, which is related to the distribution of the zeros of the scattering data. It plays a decisive role in the analysis of the corresponding RH problem and the establishment of the residue condition. Assuming that $\phi(x,t;z)$ is the solution of the first formula for Eq.\eqref{3.2a}, so is $\sigma_{2}\phi^{*}(x,t;z^{*})\sigma_{2}$, where $\sigma_{2}=\left(
                                                               \begin{array}{cc}
                                                                 0 & -i \\
                                                                 i & 0 \\
                                                               \end{array}
                                                             \right).
$ It follows from asymptotic behavior \eqref{3.1} that
\begin{align}\label{3.7}
\phi(x,t;z)=\sigma_{2}\phi^{*}(x,t;z)\sigma_{2},\quad z\in \mathbb{R},
\end{align}
from which together with \eqref{3.5}, we obtain the  symmetry of the scattering matrix
\begin{align}\label{3.8}
S(z)=\sigma_{2}S^{*}(z^{*})\sigma_{2},
\end{align}
which leads to
\begin{align}\label{3.9}
s_{11}(z)=s_{22}^{*}(z^{*}),\quad s_{12}(z)=-s_{21}^{*}(z^{*}).
\end{align}
According to the symmetry of scattering data, we further introduce the reflection coefficient and its symmetry
\begin{align}\label{3.10}
\gamma(z)=\frac{s_{21}(z)}{s_{11}(z)},\quad \frac{s_{12}(z)}{s_{22}(z)}=-\gamma^{*}(z),~z\in \mathbb{R}.
\end{align}

In order to ensure that the discrete spectrum is finite, and to avoid the need to deal with many pathologies possible when discussing general initial value problems, we will make the following assumptions
\begin{assum} \label{ass1} The initial value $q(x,0)$ for the Hirota equation satisfies the following assumptions:\\
($I$) The scattering data $s_{11}(z)\neq0$ for all $z\in \mathbb{R}$.\\
($II$) The scattering data $s_{11}(z)$ only is of finite number of simple zeros, namely, assume that the set $\mathcal{Z}=\left\{z_{j}\right\}_{j=1}^{N}$ is composed of $N$ zeros of scattering data $s_{11}(z)$. From the symmetry Eq.\eqref{3.9}, the set $\mathcal{\overline{Z}}=\left\{z_{j}^{*}\right\}_{j=1}^{N}$ is composed of $N$ zeros of scattering data $s_{22}(z)$.
\end{assum}

To get a suitable RH problem, we introduce piecewise meromorphic function $M(z)=M(x,t;z)$
\begin{align}\label{3.11}
M(x,t;z)=\left\{\begin{aligned}
 \left(\frac{\phi_{-,1}(x,t;z)}{s_{11}(z)},\phi_{+,2}(x,t;z)\right), \quad z\in\mathbb{C}^{+},\\
 \left(\phi_{+,1}(x,t;z),\frac{\phi_{-,2}(x,t;z)}{s_{22}(z)}\right), \quad z\in \mathbb{C}^{-},
\end{aligned}\right.,
\end{align}
which satisfies the following RH problem.
\begin{RHP}\label{RHP1}
Find an analytic function $M(z)=M(x,t;z)$ that satisfies the following properties:\\
($I$) $M(z)$ is analytic in $C\setminus(R\cup\mathcal {Z}\cup\mathcal {\overline{Z}})$.\\
($II$) The jump condition:
\begin{align}\label{3.12}
M^{+}(z)=M^{-}(z)V(z),\quad z\in \mathbb{R}\triangleq \Sigma^{(1)},
\end{align}
where $M^{\pm}(x,t;z)=\lim\limits_{\varepsilon\rightarrow0^{+}}M(x,t;z\pm i\varepsilon),~\varepsilon\in \mathbb{R}$, and
\begin{align}\label{3.13}
V(z)=\left(
       \begin{array}{cc}
         1+|r(z)|^{2} & -r^{*}(z)e^{-2it\theta(z)} \\
         r(z)e^{2it\theta(z)} & 1 \\
       \end{array}
     \right),
\end{align}
where $\theta(z)=z\frac{x}{t}+2\alpha z^{2}+4\beta z^{3}$.\\
($III$) Asymptotic behavior: $M(z)=\mathbb{I}+O(z^{-1})$ as $z\rightarrow\infty$.\\
($IV$) For any $z_{j}\in\mathcal {Z}$ and $z_{j}\in\mathcal {\overline{Z}}$, the residue conditions for $M(z)$ are determined by
\begin{align}\label{3.14}
\mathop{Res}_{z=z_{j}}M=\lim_{z\rightarrow z_{j}}M\left(\begin{array}{cc}
                   0 & 0 \\
                   c_{j}e^{2it\theta(z)} & 0
                 \end{array}\right),~~
\mathop{Res}_{z=z^{*}_{j}}M=\lim_{z\rightarrow z^{*}_{j}}M\left(\begin{array}{cc}
                   0 & -c^{*}_{j}e^{-2it\theta(z)} \\
                   0 & 0
                 \end{array}\right).
\end{align}
\end{RHP}
\begin{rem}
The Zhou's vanishing lemma \cite{DZ-1989} can guarantee the existence of the solution of the \textbf{RH Problem} \ref{RHP1} by transforming the poles into jumps on a sufficiently small circular contours. The uniqueness of the solution is a natural result of Liouville's theorem when the solution exists.
\end{rem}

Expanding $M(z)$ as follows
\begin{align}\label{3.15}
M(z)=\mathbb{I}+\frac{M_{1}(z)}{z}+\mathcal {O}(z^{-1}),\quad z\rightarrow\infty,
\end{align}
and substituting it into Eq.\eqref{3.2a}, we have the following reconstruction formula of $q(x,t)$
\begin{align}\label{3.16}
q(x,t)=2i\lim_{z\rightarrow\infty}(zM(z))_{12},
\end{align}
by comparing the coefficients of $z^{i}$ ($i=0,1,2,3$), where $A_{ij}$ denote the $i-th$ row  and  $j-th$ column of the matrix $A$.

\section{Conjugation}
In order to study the long-time asymptotic behavior of solutions, we first deal with the oscillatory term in the jump matrix defined in Eq.\eqref{3.13}, that is, we decompose the jump matrix according to the sign change graph of $Re(it\theta)$ a to ensure that any jump matrix is bounded in a given region. In this work, we only consider the case $\beta>0$ for our analysis convenient.

As shown in \cite{DZ-AM}, the stationary points of the function $\theta(z)$ should be considered
\begin{align}\label{4.1}
\theta(z)=z\frac{x}{t}+2\alpha z^{2}+4\beta z^{3}.
\end{align}
Letting
\begin{align*}
\frac{d(\theta(z))}{dz}=0,
\end{align*}
which leads to
\begin{align}\label{4.2}
z_{0}=\frac{-\alpha-\sqrt{\alpha^{2}-3\beta x/t}}{6\beta},\quad
z_{1}=\frac{-\alpha+\sqrt{\alpha^{2}-3\beta x/t}}{6\beta}.
\end{align}
The sign change diagram of the function $Re(it\theta(z))$ is shown in Fig. 2.

\centerline{\begin{tikzpicture}[scale=0.9]
\draw[-][thick](-4,0)--(-3,0);
\draw[-][thick](-3,0)--(-2,0);
\draw[-][thick](-2.0,0)--(-1.0,0)node[below left]{$z_{0}$};
\draw[-][thick](-1,0)--(0,0);
\draw[-][thick](0,0)--(1,0)node[below right]{$z_{1}$};
\draw[-][thick](1,0)--(2,0);
\draw[fill] (1,0) circle [radius=0.035];
\draw[fill] (-1,0) circle [radius=0.035];
\draw[-][thick](2.0,0)--(3.0,0);
\draw[-][thick](3.0,0)--(4.0,0);
\draw [-,thick, cyan] (1,0) to [out=90,in=-120] (1.9,3);
\draw [-,thick, cyan] (1,0) to [out=-90,in=120] (1.9,-3);
\draw [-,thick, cyan] (-1,0) to [out=90,in=-60] (-1.9,3);
\draw [-,thick, cyan] (-1,0) to [out=-90,in=60] (-1.9,-3);
\draw[fill] (0,2) node{$Re(it\theta(z))>0$};
\draw[fill] (0,-2) node{$Re(it\theta(z))<0$};
\draw[fill] (-2.5,1) node{$Re(it\theta(z))<0$};
\draw[fill] (2.5,1) node{$Re(it\theta(z))<0$};
\draw[fill] (2.5,-1) node{$Re(it\theta(z))>0$};
\draw[fill] (-2.5,-1) node{$Re(it\theta(z))>0$};
\end{tikzpicture}}
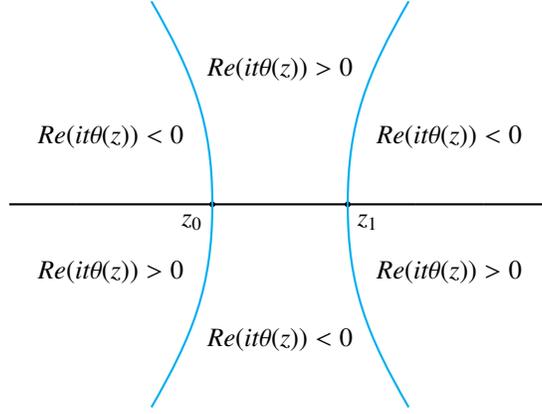
\centerline{\noindent {\small \textbf{Figure 2.} (Color online) The signature table for $Re( it\theta(z))$ in the complex $z$-plane.}}

For $z\in(z_{0},z_{1})$, the upper and lower triangular decomposition of jump matrix $V(z)$ will produce a diagonal matrix, therefore similar to \cite{Huang-NA}  we first introduce a scalar RH problem about $\delta(z)$ to deal with this case
\begin{align}\label{4.3}
\left\{\begin{aligned}
&\delta_{+}(z)=\delta_{-}(z)(1+|\gamma(z)|^{2}),\quad z\in(z_{0},z_{1}),\\
&\delta_{+}(z)=\delta_{-}(z)=\delta(z),\qquad\quad \mathbb{R}\setminus(z_{0},z_{1}),\\
&\delta(z)\rightarrow \mathbb{I},\quad z\rightarrow\infty.
\end{aligned}\right.
\end{align}
Then the solution of Eq.\eqref{4.3} can be written as
\begin{align}\label{4.4}
\delta(z)=exp\left[\frac{1}{2\pi i}\int_{z_{0}}^{z_{1}}\frac{\log(1+|\gamma|^{2})}
{s-z}ds\right]\triangleq  exp\left(i\int_{z_{0}}^{z_{1}}\frac{\nu(s)}{s-z}ds\right),
\end{align}
where $\nu(z)=-\frac{1}{2\pi}\log(1+|\gamma(z)|^{2})$.

For the convenience of the following analysis, we introduce the following notation
\begin{align}
&\triangle_{z_{1}}^{-}=\left\{k\in\{1,2,\cdots,N\},~z_{0}<z_{k}<z_{1}\right\},\label{4.5}\\
&\triangle_{z_{1}}^{+}=\left\{k\in\{1,2,\cdots,N\},~z_{0}<z_{k}, z_{k}>z_{1}\right\},\label{4.6}\\
&T(z)=T(z,z_{1})=\prod_{z\in\triangle_{z_{1}}^{-}}\frac{z-z_{k}^{*}}{z-z_{k}}\cdot \delta(z),\label{4.7}\\
&T_{0}(z_{1})=T(z_{1},z_{1})=\prod_{z\in\triangle_{z_{1}}^{-}}\frac{z_{1}-z_{k}^{*}}
{z_{1}-z_{k}}\cdot e^{i\beta^{+}(z_{1},z_{1})},\label{4.8}\\
&T_{0}(z_{0})=T(z_{0},z_{1})=\prod_{z\in\triangle_{z_{1}}^{-}}\frac{z_{0}-z_{k}^{*}}
{z_{0}-z_{k}}\cdot e^{i\beta^{-}(z_{0},z_{1})},\label{4.9}\\
&\beta^{+}(z,z_{1})=-\nu(z_{1})\ln(z-z_{1}+1)+\int_{z_{0}}^{z_{1}}\frac{\nu(s)-\chi_{1}(s)
\nu(z_{1})}{s-z}ds,\label{4.10}\\
&\beta^{-}(z,z_{0})=-\nu(z_{0})\ln(z-z_{0}+1)+\int_{z_{0}}^{z_{1}}\frac{\nu(s)-\chi_{0}(s)
\nu(z_{0})}{s-z}ds,\label{4.11}\\
&\chi_{1}(s)=\left\{\begin{aligned}
1, \quad z_{1}<s<z_{1}+1,\\
0,\quad\quad\quad \text{elsewhere},
\end{aligned}
\right.,\quad
\chi_{0}(s)=\left\{\begin{aligned}
1, \quad z_{0}<s<z_{0}+1,\\
0, \quad\quad\quad\text{elsewhere}.
\end{aligned}
\right..\label{4.12}
\end{align}
\begin{prop}\label{prop3}
The function $T(z)$ defined by Eq.\eqref{4.7} such that \\
$(I)$ $T(z)$ is meromorphic function in $\mathbb{C}\setminus(z_{0},z_{1})$, and has a simple pole at $z_{n}$ and a simple zero at $z_{n}^{*}$ for $n\in\bigtriangleup_{z_{1}}^{-}$.\\
$(II)$ For $z\in \mathbb{C}\setminus(z_{0},z_{1})$, $T^{*}(z^{*})T(z)=1$.\\
$(III)$ For $z\in \mathbb{C}\setminus(z_{0},z_{1})$, $T(z)$ is of the jump condition
\begin{align*}
T_{+}(z)=T_{-}(z)(1+|\gamma(z)|^{2}), ~z\in(z_{0},z_{1}),
\end{align*}
where $T_{\pm}(z)=\lim\limits_{\varepsilon\rightarrow0^{+}}T(z\pm i\varepsilon)$.\\
$(IV)$  As $|z|\rightarrow \infty $ with $|arg(z)|\leq c<\pi$,
\begin{align}\label{4.13}
T(z)=1+\frac{i}{z}\left[2\sum_{k\in\Delta_{z_{0}}^{-}}Im(z_{k})-
\int_{z_{0}}^{z_{1}}\nu(s)ds\right]+O(z^{-2}).
\end{align}\\
$(V)$ As $z\rightarrow z_{0}$ along any ray $z_{0}+e^{i\phi}R_{+}$  and $z\rightarrow z_{1}$  along any ray $z_{1}+e^{i\phi}R_{+}$  with $|\phi|\leq c<\pi$
\begin{align}
|T(z-z_{1})-T_{0}(z_{1})(z-z_{1})^{i\nu(z_{1})}|\leq C\parallel r\parallel_{H^{1}(R)}|z-z_{1}|^{\frac{1}{2}},\label{4.14}\\
|T(z-z_{0})-T_{0}(z_{0})(z-z_{0})^{i\nu(z_{0})}|\leq C\parallel r\parallel_{H^{1}(R)}|z-z_{0}|^{\frac{1}{2}},\label{4.15}
\end{align}
\end{prop}

In order to redistribute the poles and eliminate the diagonal matrix generated by the matrix triangular decomposition of the jumping matrix, we employ $T(z)$ to define a new matrix function $M^{(1)}(z)$ such that
\begin{align}\label{4.16}
M^{(1)}(z)=M(z)T(z)^{-\sigma_{3}},
\end{align}
which admits the RH problem
\begin{RHP}\label{RHP2}
The new matrix $M^{(1)}(z)$ such that:\\
$(I)$ Analyticity: $M^{(1)}(z)$ is analytic in $z\in \mathbb{C}\setminus\left\{\mathbb{R}
\cup\mathcal {Z}\cup\mathcal {Z}^{*}\right\}$.\\
$(II)$ Jump condition: as $z$ approaches the real axis from above and below, $M^{(1)}(z)$ satisfies
\begin{align}
M_{+}^{(1)}(z)=M_{-}^{(1)}(z)V^{(1)}(z), \quad z\in \mathbb{R},
\end{align}
where \begin{align}\label{4.17}
       V^{(1)}=\left\{\begin{aligned}
      &\left(
        \begin{array}{cc}
      1 & \gamma^{*}(z)T(z)^{2}e^{-2it\theta} \\
      0 & 1 \\
        \end{array}
      \right)\left(
     \begin{array}{cc}
       1 & 0 \\
       \gamma(z)T(z)^{-2}e^{2it\theta} & 1 \\
      \end{array}
    \right),\quad z\in C\setminus \Gamma,\\
   &\left(
    \begin{array}{cc}
    1 & 0 \\
    \frac{\gamma(z)T_{-}^{-2}(z)}{1+|\gamma(z)|^{2}}e^{2it\theta} & 1 \\
     \end{array}
   \right)\left(
    \begin{array}{cc}
    1 & \frac{\gamma^{*}(z)T_{+}^{2}(z)}{1+|\gamma(z)|^{2}}e^{-2it\theta} \\
    0 & 1 \\
   \end{array}
  \right),\qquad z\in\Gamma.
   \end{aligned}\right.
   \end{align}\\
$(III)$  Asymptotic behaviours:
\begin{align}\label{4.18}
M^{(1)}(z)=\mathbb{I}+\mathcal {O}(z^{-1}),\quad z\rightarrow\infty.
\end{align}
$(IV)$  The residue conditions
\begin{align}\label{4.19}
\begin{split}
&\mathop{Res}\limits_{z=z_{k}}M^{(1)}=\left\{\begin{aligned}
&\lim_{z\rightarrow z_{k}}M^{(1)}\left(\begin{array}{cc}
    0 & c_{k}^{-1}\left((\frac{1}{T})'(z_{k})\right)^{-2}e^{-2it\theta}\\
    0 & 0 \\
  \end{array}
\right),~k\in \Delta_{z_{0}}^{-},\\
&\lim_{z\rightarrow z_{k}}M^{(1)}\left(
  \begin{array}{cc}
    0 & 0 \\c_{k}^{-1}T^{-2}(z_{k})e^{2it\theta} & 0 \\
  \end{array}\right),~~~~~~~~~~k\in \Delta_{z_{0}}^{+},
\end{aligned}\right.\\
&\mathop{Res}\limits_{z=z^{*}_{k}}M^{(1)}=\left\{\begin{aligned}
&\lim_{z\rightarrow z^{*}_{k}}M^{(1)}\left(\begin{array}{cc}
    0 & 0\\
    -\bar{c}_{k}^{-1}(T'(z^{*}_{k}))^{-2}e^{2it\theta^{*}(z)} & 0 \\
  \end{array}
\right),~~~k\in \Delta_{z_{0}}^{-},\\
&\lim_{z\rightarrow z^{*}_{k}}M^{(1)}\left(
  \begin{array}{cc}
    0 & -c^{*}_{k}(T(z^{*}_{k}))^{2}e^{-2it\theta^{*}(z)} \\0 & 0 \\
  \end{array}\right),~~~~~~k\in \Delta_{z_{0}}^{+}.
\end{aligned}\right.
\end{split}
\end{align}
\end{RHP}

\section{A mixed $\bar{\partial}$-problem}
We now will deform the jump contour of \textbf{RH Problem} \ref{RHP2} so that the oscillation term in the jump matrix is bounded on the transformed contour. It follows that from Eq.\eqref{4.2}, the phase function Eq.\eqref{4.1} has two phase points at $z_{0}$ and $z_{1}$. As usual, the new contour is chosen as folllows
\begin{align}\label{5.1}
\Sigma^{(2)}=\Sigma_{1}\cup \Sigma_{2} \cup\Sigma_{3} \cup\Sigma_{4} \cup\Sigma_{5} \cup\Sigma_{6}
\cup\Sigma_{7}\cup \Sigma_{8},
\end{align}
and consists of the rays $z_{0}+e^{i\phiup}R^{+}$ and $z_{1}+e^{i\phiup}R^{+}$ with $\phiup=\frac{\pi}{4},\frac{3\pi}{4},\frac{5\pi}{4},\frac{7\pi}{4}$. Then the contour $\Sigma^{(2)}$ and the real axis $R$ separate the complex into ten sectors denoted by $\Omega_{j},$ $(j=1,2,\cdots,6,7^{\pm},8^{\pm})$  shown in Fig. 3.

\centerline{\begin{tikzpicture}[scale=0.6]
\draw[->][thick](2,0)--(4,-2);
\draw[->][thick](2,0)--(4,2);
\draw[-][thick](0,2)[thick]node[above]{$\Omega_{2}$}--(5,-3)[thick]node[right]{$\Sigma_{4}$};
\draw[->][thick](0,2)--(1,1)[thick]node[right]{$\Sigma_{6}$};
\draw[->][thick](0,-2)[thick]node[below]{$\Omega_{5}$}--(1,-1)[thick]node[right]{$\Sigma_{8}$};
\draw[->][thick](-5,-3)--(-4,-2);
\draw[->][thick](-5,3)--(-4,2);
\draw[->][thick](-2,0)--(-1,1)[thick]node[left]{$\Sigma_{5}$};
\draw[->][thick](-2,0)--(-1,-1)[thick]node[left]{$\Sigma_{7}$};
\draw[-][thick](0,-2)--(5,3)[thick]node[right]{$\Sigma_{1}$};
\draw[-][thick](0,2)--(-5,-3)[thick]node[left]{$\Sigma_{3}$};
\draw[-][thick](0,-2)--(-5,3)[thick]node[left]{$\Sigma_{2}$};
\draw[-][dashed](-8,0)--(-7,0);
\draw[-][dashed](-7,0)--(-6,0);
\draw[-][dashed](-6,0)--(-5,0);
\draw[-][dashed](-5,0)--(-4,0);
\draw[-][dashed](-4,0)--(-3,0);
\draw[-][dashed](-3,0)--(-2,0);
\draw[-][dashed](-2,0)--(-1,0);
\draw[-][dashed](-1,0)--(0,0);
\draw[-][dashed](0,0)--(1,0);
\draw[-][dashed](1,0)--(2,0);
\draw[-][dashed](2,0)--(3,0);
\draw[-][dashed](3,0)--(4,0);
\draw[-][dashed](4,0)--(5,0);
\draw[-][dashed](5,0)--(6,0);
\draw[-][dashed](6,0)--(7,0);
\draw[-][dashed](7,0)--(8,0);
\draw[-][dashed](0,2)--(0,1);
\draw[-][dashed](0,1)--(0,0);
\draw[-][dashed](0,0)--(0,-1);
\draw[-][dashed](0,-1)--(0,-2);
\draw[fill] (5,1) node{$\Omega_{1}$};
\draw[fill] (5,-1) node{$\Omega_{6}$};
\draw[fill] (0.7,0.6) node{$\Omega_{7^{+}}$};
\draw[fill] (0.7,-0.6) node{$\Omega_{8^{+}}$};
\draw[fill] (-0.6,0.6) node{$\Omega_{7^{-}}$};
\draw[fill] (-0.6,-0.6) node{$\Omega_{8^{-}}$};
\draw[fill] (-5,1) node{$\Omega_{3}$};
\draw[fill] (-5,-1) node{$\Omega_{4}$};
\end{tikzpicture}}
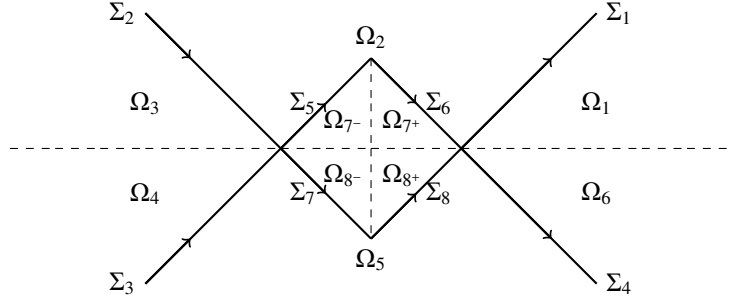
\centerline{\noindent {\small \textbf{Figure 3.} The contour of deformation from $\mathbb{R}$ to $\Sigma^{(2)}$}.}

For technical reasons, we define
\begin{align}\label{5.2}
\Xi_{\mathcal {Z}}(z)=\left\{\begin{aligned}
&1,~~dist(z,\mathcal{Z}\cup \mathcal{Z}^{*})<\rho/3, \\
&0,~~dist(z,\mathcal{Z}\cup \mathcal{Z}^{*})>2\rho/3.
\end{aligned}\right.
\end{align}
where \begin{align}\label{5.3}
\rho=\frac{1}{2}\min_{\lambda,\zeta\in \mathcal{Z}\cup \mathcal{Z}^{*} \lambda\neq\mu}|\lambda-\zeta|.
\end{align}
Note that the discrete spectrum appear as conjugate pairs from Eq.\eqref{3.9}, moreover one has $dist(\mathcal{Z}\cup \mathcal{Z}^{*}, R)>\rho$  from the Assumption \eqref{ass1}.

Now we introduce a new matrix $M^{(2)}$ such that the jump contour of the \textbf{RH Problem} \ref{RHP2} is transformed from $\Sigma^{(1)}$ to $\Sigma^{(2)}$
\begin{align}\label{5.4}
M^{(2)}(z)=M^{(1)}(z)R^{(2)}(z),
\end{align}
\begin{rem}
The selection of matrix $R^{(2)}(z)$ shown in Fig.4 and Fig.5 needs to meet the following conditions:\\
$\blacktriangleright$ To remove the jump on the real axis and the new analytic jump matrix has the expected exponential decay along the contour $\Sigma^{(2)}$.\\
$\blacktriangleright$ The norm of $R^{(2)}(z)$ should be controlled to ensure that the long-time asymptotic behavior of $\overline{\partial}$-contribution to the solution $q(x,t)$ is negligible.\\
$\blacktriangleright$ The residues are not affected by this transformation.
\end{rem}

\centerline{\begin{tikzpicture}[scale=0.6]
\draw[-][dashed](-7,0)--(-6,0);
\draw[-][dashed](-6,0)--(-5,0);
\draw[-][dashed](-5,0)--(-4,0);
\draw[-][dashed](-4,0)--(-3,0);
\draw[-][dashed](-3,0)--(-2,0);
\draw[-][dashed](-2,0)--(-1,0);
\draw[-][dashed](-1,0)--(0,0);
\draw[-][dashed](0,0)--(1,0);
\draw[-][dashed](1,0)--(2,0);
\draw[-][dashed](2,0)--(3,0);
\draw[-][dashed](3,0)--(4,0);
\draw[-][dashed](4,0)--(5,0);
\draw[-][dashed](5,0)--(6,0);
\draw[-][dashed](6,0)--(7,0);
\draw[-][thick](-4,-4)--(4,4);
\draw[-][thick](-4,4)--(4,-4);
\draw[->][thick](2,2)--(3,3);
\draw[->][thick](-4,4)--(-3,3);
\draw[->][thick](-4,-4)--(-3,-3);
\draw[->][thick](2,-2)--(3,-3);
\draw[fill] (3.2,3)node[below right]{$\Sigma_{1}$};
\draw[fill] (3.2,-3)node[above right]{$\Sigma_{4}$};
\draw[fill] (-3.2,3)node[below]{$\Sigma_{6}$};
\draw[fill] (-3.2,-3)node[above left]{$\Sigma_{8}$};
\draw[fill] (0,0)node[below]{$z_{1}$};
\draw[fill] (1,0)node[below right]{$\Omega_{6}$};
\draw[fill] (1,0)node[above right]{$\Omega_{1}$};
\draw[fill] (0,-1)node[below]{$\Omega_{5}$};
\draw[fill] (0,1)node[above]{$\Omega_{2}$};
\draw[fill] (-1,0)node[below left]{$\Omega_{8^{+}}$};
\draw[fill] (-1,0)node[above left]{$\Omega_{7^{+}}$};
\draw[fill] (7,3)node[below]{$
\left(
  \begin{array}{cc}
    1 & 0 \\
    -R_{1}e^{2it\theta} & 1 \\
  \end{array}
\right)
$};
\draw[fill] (7,-2)node[below]{$
\left(
  \begin{array}{cc}
    1 & R_{6}e^{-2it\theta} \\
    0 & 1 \\
  \end{array}
\right)
$};
\draw[fill] (-7,2.5)node[below]{$
\left(
  \begin{array}{cc}
    1 & -R_{7^{+}}e^{-2it\theta} \\
    0 & 1 \\
  \end{array}
\right)
$};
\draw[fill] (-7,-1)node[below]{$
\left(
  \begin{array}{cc}
    1 & 0 \\
    R_{8^{+}}e^{2it\theta} & 1 \\
  \end{array}
\right)
$};
\draw[fill] (0,2)node[above]{$
\left(
  \begin{array}{cc}
    1 & 0 \\
    0 & 1 \\
  \end{array}
\right)
$};
\draw[fill] (0,-2)node[below]{$
\left(
  \begin{array}{cc}
    1 & 0 \\
    0 & 1 \\
  \end{array}
\right)
$};
\end{tikzpicture}}
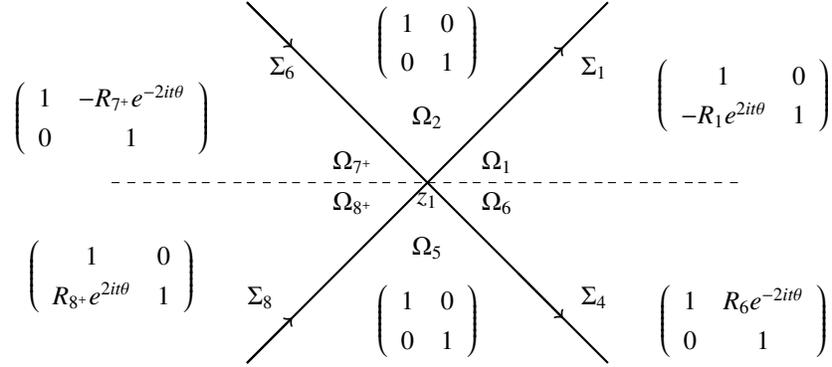
\centerline{\noindent {\small \textbf{Figure 4.} The matrix $R^{(2)}(z)$ near $z_{1}$}.}

\centerline{\begin{tikzpicture}[scale=0.6]
\draw[-][dashed](-7,0)--(-6,0);
\draw[-][dashed](-6,0)--(-5,0);
\draw[-][dashed](-5,0)--(-4,0);
\draw[-][dashed](-4,0)--(-3,0);
\draw[-][dashed](-3,0)--(-2,0);
\draw[-][dashed](-2,0)--(-1,0);
\draw[-][dashed](-1,0)--(0,0);
\draw[-][dashed](0,0)--(1,0);
\draw[-][dashed](1,0)--(2,0);
\draw[-][dashed](2,0)--(3,0);
\draw[-][dashed](3,0)--(4,0);
\draw[-][dashed](4,0)--(5,0);
\draw[-][dashed](5,0)--(6,0);
\draw[-][dashed](6,0)--(7,0);
\draw[-][thick](-4,-4)--(4,4);
\draw[-][thick](-4,4)--(4,-4);
\draw[->][thick](2,2)--(3,3);
\draw[->][thick](-4,4)--(-3,3);
\draw[->][thick](-4,-4)--(-3,-3);
\draw[->][thick](2,-2)--(3,-3);
\draw[fill] (3.2,3)node[below right]{$\Sigma_{5}$};
\draw[fill] (3.2,-3)node[above right]{$\Sigma_{7}$};
\draw[fill] (-3.2,3)node[below]{$\Sigma_{2}$};
\draw[fill] (-3.2,-3)node[above left]{$\Sigma_{3}$};
\draw[fill] (0,0)node[below]{$z_{0}$};
\draw[fill] (1,0)node[below right]{$\Omega_{8^{-}}$};
\draw[fill] (1,0)node[above right]{$\Omega_{7^{-}}$};
\draw[fill] (0,-1)node[below]{$\Omega_{5}$};
\draw[fill] (0,1)node[above]{$\Omega_{2}$};
\draw[fill] (-1,0)node[below left]{$\Omega_{4}$};
\draw[fill] (-1,0)node[above left]{$\Omega_{3}$};
\draw[fill] (7,3)node[below]{$
\left(
  \begin{array}{cc}
    1 & -R_{7^{-}}e^{-2it\theta} \\
    0 & 1 \\
  \end{array}
\right)
$};
\draw[fill] (7,-2)node[below]{$
\left(
  \begin{array}{cc}
    1 & 0 \\
    R_{8^{-}}e^{2it\theta} & 1 \\
  \end{array}
\right)
$};
\draw[fill] (-7,2.5)node[below]{$
\left(
  \begin{array}{cc}
    1 & 0 \\
    R_{3}e^{2it\theta} & 1 \\
  \end{array}
\right)
$};
\draw[fill] (-7,-1)node[below]{$
\left(
  \begin{array}{cc}
    1 & R_{4}e^{-2it\theta} \\
    0 & 1 \\
  \end{array}
\right)
$};
\draw[fill] (0,2)node[above]{$
\left(
  \begin{array}{cc}
    1 & 0 \\
    0 & 1 \\
  \end{array}
\right)
$};
\draw[fill] (0,-2)node[below]{$
\left(
  \begin{array}{cc}
    1 & 0 \\
    0 & 1 \\
  \end{array}
\right)
$};
\end{tikzpicture}}
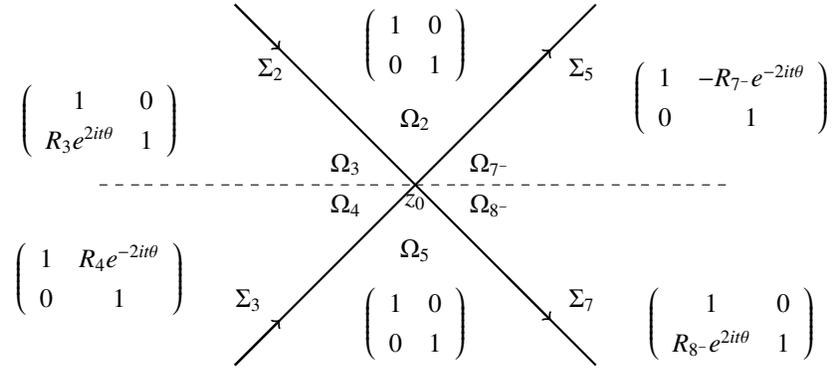
\centerline{\noindent {\small \textbf{Figure 5.} The matrix $R^{(2)}(z)$ near $z_{0}$}.}

The matrices $R_{j}$ ($j=1,3,4,6,7^{\pm},8^{\pm}$) defined in Fig.4 and Fig.5 such that
\begin{prop}\label{prop4}
The matrices $R_{j}:$ $\bar{\Omega}\rightarrow C$ have the following boundary values:\\
\begin{align}
&R_{1}(z)=\left\{\begin{aligned}
&\gamma(z)T^{-2}(z),\qquad\qquad\qquad\qquad\qquad z\in(z_{1},\infty),\\
&\gamma(z_{1})T_{0}^{-2}(z_{1})(z-z_{1})^{-2i\nu(z_{1})}(1-\Xi_{\mathcal {Z}_{1}}),~~z\in\Sigma_{1},
\end{aligned}\right. \label{5.5} \\
&R_{6}(z)=\left\{\begin{aligned}
&\gamma^{*}(z)T^{2}(z),\qquad\qquad\qquad\qquad\qquad z\in(z_{1},\infty),\\
&\gamma^{*}(z_{1})T_{0}^{2}(z_{1})(z-z_{1})^{2i\nu(z_{1})}(1-\Xi_{\mathcal {Z}_{1}}),~~z\in\Sigma_{4},
\end{aligned}\right. \label{5.6} \\
&R_{3}(z)=\left\{\begin{aligned}
&\gamma(z)T^{-2}(z),\qquad\qquad\qquad\qquad\qquad z\in(-\infty,z_{0}),\\
&\gamma(z_{0})T_{0}^{-2}(z_{0})(z-z_{0})^{-2i\nu(z_{0})}(1-\Xi_{\mathcal {Z}_{0}}),~~z\in\Sigma_{2},
\end{aligned}\right. \label{5.7} \\
&R_{4}(z)=\left\{\begin{aligned}
&\gamma^{*}(z)T^{2}(z),\qquad\qquad\qquad\qquad\qquad z\in(-\infty,z_{0}),\\
&\gamma^{*}(z_{0})T_{0}^{2}(z_{0})(z-z_{0})^{2i\nu(z_{0})}(1-\Xi_{\mathcal {Z}_{0}}),~~z\in\Sigma_{3},
\end{aligned}\right. \label{5.8}  \\
&R_{7^{+}}(z)=\left\{\begin{aligned}
&\frac{\gamma^{*}(z)T_{+}^{-2}(z)}{1+|\gamma(z)|^{2}},\qquad\qquad\qquad\qquad\qquad z\in(z_{0},z_{1}),\\
&\frac{\gamma^{*}(z_{1})T_{0}^{-2}(z_{1})}{1+|\gamma(z_{1})|^{2}}(z-z_{1})^{-2i\nu(z_{1})}
(1-\Xi_{\mathcal {Z}_{1}}),~~z\in\Sigma_{6},
\end{aligned}\right.  \label{5.9} \\
&R_{8^{+}}(z)=\left\{\begin{aligned}
&\frac{\gamma (z)T_{-}^{-2}(z)}{1+|\gamma(z)|^{2}},\qquad\qquad\qquad\qquad\qquad z\in(z_{0},z_{1}),\\
&\frac{\gamma(z_{1})T_{0}^{-2}(z_{1})}{1+|\gamma(z_{1})|^{2}}(z-z_{1})^{-2i\nu(z_{1})}
(1-\Xi_{\mathcal {Z}_{1}}),~~z\in\Sigma_{8},
\end{aligned}\right. \label{5.10} \\
&R_{7^{-}}(z)=\left\{\begin{aligned}
&\frac{\gamma^{*}(z)T_{+}^{2}(z)}{1+|\gamma(z)|^{2}},\qquad\qquad\qquad\qquad\qquad z\in(z_{0},z_{1}),\\
&\frac{\gamma^{*}(z_{0})T_{0}^{2}(z_{0})}{1+|\gamma(z_{0})|^{2}}(z-z_{0})^{2i\nu(z_{0})}
(1-\Xi_{\mathcal {Z}_{0}}),~~z\in\Sigma_{5},
\end{aligned}\right. \label{5.11} \\
&R_{8^{-}}(z)=\left\{\begin{aligned}
&\frac{\gamma (z)T_{-}^{-2}(z)}{1+|\gamma(z)|^{2}},\qquad\qquad\qquad\qquad\qquad z\in(z_{0},z_{1}),\\
&\frac{\gamma(z_{0})T_{0}^{-2}(z_{0})}{1+|\gamma(z_{0})|^{2}}(z-z_{0})^{-2i\nu(z_{0})}
(1-\Xi_{\mathcal {Z}_{0}}),~~z\in\Sigma_{7},
\end{aligned}\right. \label{5.12}
\end{align}
These matrices $R_{j}$ are estimated as follows
\begin{align}
&|R_{j}|\lesssim \sin^{2}(arg(z-z_{1}))+\langle Re(z)\rangle^{-\frac{1}{2}}, \label{5.13}\\
&|\bar{\partial}R_{j}(z)|\lesssim \left|\bar{\partial}\Xi_{1}(z)\right|+ \left|p'_{j}(Re(z))\right|+|z-z_{1}|^{-1/2},\label{5.14}
\end{align}
for $j=1,6,7^{+},8^{+}$, similarly for $j=3,4,7^{-},8^{-}$, we have
\begin{align}
&|R_{j}|\lesssim \sin^{2}(arg(z-z_{0}))+\langle Re(z)\rangle^{-\frac{1}{2}}, \label{5.15}\\
&|\bar{\partial}R_{j}(z)|\lesssim \left|\bar{\partial}\Xi_{0}(z)\right|+ \left|p'_{j}(Re(z))\right|+|z-z_{0}|^{-1/2},\label{5.16}
\end{align}
where
\begin{align}
&p_{1}(z)=p_{3}(z)=\gamma(z),\quad p_{7^{+}}(z)=p_{7^{-}}(z)=\frac{\gamma^{*}(z)}{1+|\gamma(z)|^{2}},\label{5.17}\\
&p_{4}(z)=p_{6}(z)=\gamma^{*}(z),\quad p_{8^{+}}(z)=p_{8^{-}}(z)=\frac{\gamma(z)}{1+|\gamma(z)|^{2}},\label{5.18}\\
&\langle (\cdot)\rangle=\sqrt{1+(\cdot)^{2}}.\label{5.19}
\end{align}
\end{prop}

Employing $M^{(2)}(z)$ defined by $R^{(2)}(z)$, we can get a new RH problem.
\begin{RHP}\label{RHP3}
The matrix-valued function $M^{(2)}(z)$ such that:\\
$(I)$ Analyticity: $M^{(2)}(z)$ is continuous in complex $\mathbb{C}$, sectionally continuous first partial derivatives in $\mathbb{C}\setminus(\Sigma^{(2)}\cup\mathcal {Z}\cup\overline{\mathcal {Z}})$.\\
$(II)$ Jump condition:
\begin{align}\label{5.20}
M^{(2)}_{+}(z)=M^{(2)}_{-}(z)V^{(2)}(z)
\end{align}
where  the jump matrix
\begin{align*}
V^{(2)}(z)=\left\{\begin{aligned}
&\left(
  \begin{array}{cc}
    1 & 0 \\
    R_{1}e^{2it\theta} & 1 \\
  \end{array}
\right), \quad z\in\Sigma_{1},\\
&\left(
  \begin{array}{cc}
    1 & R_{7^{+}}e^{-2it\theta} \\
    0 & 1 \\
  \end{array}
\right), \quad z\in\Sigma_{6},\\
&\left(
  \begin{array}{cc}
    1 & 0 \\
    R_{8^{+}}e^{2it\theta} & 1 \\
  \end{array}
\right), \quad z\in\Sigma_{8},\\
&\left(
  \begin{array}{cc}
    1 &  R_{6}e^{-2it\theta} \\
    0 & 1 \\
  \end{array}
\right), \quad z\in\Sigma_{4},\\
&\left(
  \begin{array}{cc}
    1 &  R_{7^{-}}e^{-2it\theta} \\
    0 & 1 \\
  \end{array}
\right), \quad z\in\Sigma_{5},\\
&\left(
  \begin{array}{cc}
    1 & 0 \\
    R_{3}e^{2it\theta} & 1 \\
  \end{array}
\right), \quad z\in\Sigma_{2},\\
&\left(
  \begin{array}{cc}
    1 &  R_{4}e^{-2it\theta} \\
    0 & 1 \\
  \end{array}
\right), \quad z\in\Sigma_{3},\\
&\left(
  \begin{array}{cc}
    1 & 0 \\
    R_{8^{-}}e^{2it\theta} & 1 \\
  \end{array}
\right), \quad z\in\Sigma_{7}.
\end{aligned}\right.
\end{align*}\\
$(III)$ Analyticity behavior:
\begin{align}\label{5.21}
M^{(2)}(z)=\mathbb{I}+\mathcal {O}(z^{-1}),\quad z\rightarrow\infty.
\end{align}\\
$(IV)$ The $\bar{\partial}$-derivative: For $\mathbb{C}\setminus(\Sigma^{(2)}\cup\mathcal {Z}\cup\overline{\mathcal {Z}})$
\begin{align}\label{5.22}
\bar{\partial}M^{(2)}(z)=M^{(1)}\bar{\partial}R^{(2)}(z),
\end{align}
where
\begin{align}\label{5.23}
\bar{\partial}R^{(2)}(z)=\left\{\begin{aligned}
&\left(
                                  \begin{array}{cc}
                                    1 & (-1)^{m}\bar{\partial}R_{j}(z) e^{-2it\theta}\\
                                    0 & 1 \\
                                  \end{array}
                                \right),\quad j=7^{\pm},4,6,\\
&\left(
                                  \begin{array}{cc}
                                    1 & 0\\
                                    (-1)^{m}\bar{\partial}R_{j}(z) e^{2it\theta} & 1 \\
                                  \end{array}
                                \right),\quad j=8^{\pm},1,3,\\
& \left(
                                  \begin{array}{cc}
                                    1 & 0\\
                                    0 & 1 \\
                                  \end{array}
                                \right),\qquad\qquad\qquad\quad z\in\Omega_{2}\cup\Omega_{5},
\end{aligned}\right.
\end{align}
and $m=\left\{\begin{aligned}
1, j=7^{\pm}, 1,3,\\
0, j=8^{\pm},4, 6,
\end{aligned}\right.$.\\
$(V)$  The residue conditions: $M^{(2)}$ has simple poles at the set $\mathcal{Z}\cup\overline{\mathcal {Z}}$
\begin{align}\label{5.24}
\begin{split}
&\mathop{Res}\limits_{z=z_{k}}M^{(2)}=\left\{\begin{aligned}
&\lim_{z\rightarrow z_{k}}M^{(2)}\left(\begin{array}{cc}
    0 & c_{k}^{-1}\left((\frac{1}{T})'(z_{k})\right)^{-2}e^{-2it\theta}\\
    0 & 0 \\
  \end{array}
\right),~k\in \Delta_{z_{0}}^{-},\\
&\lim_{z\rightarrow z_{k}}M^{(2)}\left(
  \begin{array}{cc}
    0 & 0 \\c_{k}^{-1}T^{-2}(z_{k})e^{2it\theta} & 0 \\
  \end{array}\right),~~~~~~~~~~k\in \Delta_{z_{0}}^{+},
\end{aligned}\right.\\
&\mathop{Res}\limits_{z=z^{*}_{k}}M^{(2)}=\left\{\begin{aligned}
&\lim_{z\rightarrow z^{*}_{k}}M^{(2)}\left(\begin{array}{cc}
    0 & 0\\
    -\bar{c}_{k}^{-1}(T'(z^{*}_{k}))^{-2}e^{2it\theta^{*}(z)} & 0 \\
  \end{array}
\right),~~~k\in \Delta_{z_{0}}^{-},\\
&\lim_{z\rightarrow z^{*}_{k}}M^{(2)}\left(
  \begin{array}{cc}
    0 & -c^{*}_{k}(T(z^{*}_{k}))^{2}e^{-2it\theta^{*}(z)} \\0 & 0 \\
  \end{array}\right),~~~~~~k\in \Delta_{z_{0}}^{+}.
\end{aligned}\right.
\end{split}
\end{align}
\end{RHP}
\section{The decomposition of the mixed $\overline{\partial}$-RH probelm}
In this section, the pure RH problem in the mixed \textbf{RH problem} \ref{RHP3} is separated from setting $\overline{\partial}R^{(2)}(z)=0$, that is, the soliton solution corresponding to the discrete spectrum is denoted as $M_{RHP}^{(2)}(z)$. We then establish the following RH problem
\begin{RHP}\label{RHP4}
Find a matrix-valued function $M^{(2)}_{RHP}(z)$ admits the following properties:\\
$(I)$ Analyticity: $M^{(2)}_{RHP}(z)$ is analytic in $\mathbb{C}\setminus(\Sigma^{(2)}\cup\mathcal {Z}\cup\overline{\mathcal {Z}})$.\\
$(II)$ Analytic behavior:
\begin{align}\label{6.1}
M^{(2)}_{RHP}(z)=\mathbb{I}+\mathcal {O}(z^{-1}),\quad z\rightarrow\infty.
\end{align}\\
$(III)$ The jump condition:
\begin{align}\label{6.2}
M_{RHP,+}^{(2)}(z)=M_{RHP,-}^{(2)}(z)V^{(2)}(z),\quad z\in \mathbb{R}.
\end{align}
$(IV)$ The $\bar{\partial}$-derivative: $\overline{\partial}R^{(2)}(z)=0$ for $z\in \mathbb{C}$.\\
$(V)$ The residue conditions: $M^{(2)}_{RHP}(z)$ has simple poles at the set $\mathcal {Z}\cup\overline{\mathcal {Z}}$.
\begin{align}\label{6.3}
\begin{split}
&\mathop{Res}\limits_{z=z_{k}}M^{(2)}_{RHP}(z)=\left\{\begin{aligned}
&\lim_{z\rightarrow z_{k}}M^{(2)}_{RHP}(z)\left(\begin{array}{cc}
    0 & c_{k}^{-1}\left((\frac{1}{T})'(z_{k})\right)^{-2}e^{-2it\theta}\\
    0 & 0 \\
  \end{array}
\right),~k\in \Delta_{z_{0}}^{-},\\
&\lim_{z\rightarrow z_{k}}M^{(2)}_{RHP}(z)\left(
  \begin{array}{cc}
    0 & 0 \\c_{k}^{-1}T^{-2}(z_{k})e^{2it\theta} & 0 \\
  \end{array}\right),~~~~~~~~~~k\in \Delta_{z_{0}}^{+},
\end{aligned}\right.\\
&\mathop{Res}\limits_{z=z^{*}_{k}}M^{(2)}_{RHP}(z)=\left\{\begin{aligned}
&\lim_{z\rightarrow z^{*}_{k}}M^{(2)}_{RHP}(z)\left(\begin{array}{cc}
    0 & 0\\
    -\bar{c}_{k}^{-1}(T'(z^{*}_{k}))^{-2}e^{2it\theta^{*}(z)} & 0 \\
  \end{array}
\right),~~~k\in \Delta_{z_{0}}^{-},\\
&\lim_{z\rightarrow z^{*}_{k}}M^{(2)}_{RHP}(z)\left(
  \begin{array}{cc}
    0 & -c^{*}_{k}(T(z^{*}_{k}))^{2}e^{-2it\theta^{*}(z)} \\0 & 0 \\
  \end{array}\right),~~~~~~k\in \Delta_{z_{0}}^{+}.
\end{aligned}\right.
\end{split}
\end{align}
\end{RHP}
The existence and asymptotic of $M^{(2)}_{RHP}(z)$ will shown in next section. In virtue of $M^{(2)}_{RHP}(z)$, a new matrix function $M^{(3)}(z)$ is constructed
\begin{align}\label{6.4}
M^{(3)}(z)=M^{(2)}(z)M^{(2)}_{RHP}(z)^{-1},
\end{align}
which satisfies a pure $\overline{\partial}$-problem.
\begin{RHP}\label{Dbar-RHP}
The matrix $M^{(3)}(z)$ defined in \eqref{6.4} such that:\\
$(I)$ $M^{(3)}(z)$ is continuous with sectionally continuous first partial derivatives in $\mathbb{C}\backslash(\Sigma^{(2)}\cup\mathcal{Z}\cup\mathcal{Z}^{*})$.\\
$(II)$ For $z\in \mathbb{C}$, one has
\begin{align}\label{6.5}
\bar{\partial}M^{(3)}(z)=M^{(3)}(z)M_{RHP}^{(2)}(z)\bar{\partial}R^{(2)}M_{RHP}^{(2)}
(z)^{-1}\triangleq M^{(3)}(z)W^{(3)}(z),
 \end{align}
where $\bar{\partial}R^{(2)}(z)$ is defined by \eqref{5.23}.\\
$(III)$  Asymptotic behavior:
 \begin{align}\label{6.6}
       M^{(3)}(z)=\mathbb{I}+\mathcal {O}(z^{-1}), \quad z\rightarrow\infty.
       \end{align}
\end{RHP}
\begin{proof}
The analytic and asymptotic properties of $M^{(2)}(z)$ and $M_{RHP}^{(2)}(z)$ are exactly given in \textbf{RH Problem} \ref{RHP3} and  \textbf{RH Problem} \ref{RHP4}. According to Eq.\eqref{6.3}, we can derive these properties of $M^{(3)}(z)$. Next, we will show that the matrix $M^{(3)}(z)$ has no  jumps through the real axis and has no poles. It follows $M^{(2)}(z)$ and $M_{RHP}^{(2)}(z)$ admit the same jump condition  that
\begin{align*}
M^{(3)}_{-}(z)^{-1}M^{(3)}_{+}(z)&=M_{RHP,-}(z)M^{(2)}_{-}(z)^{-1}M^{(2)}_{+}(z)M_{RHP,+}(z)^{-1}\\
&=M_{RHP,-}(z)V^{(2)}(z)\left(M_{RHP,-}(z)V^{(2)}(z)\right)^{-1}=\mathbb{I}.
\end{align*}
Note that for each $z_{k}\in\mathcal {Z}\cup\overline{\mathcal {Z}}$, $M^{(2)}(z)$ and $M_{RHP}^{(2)}(z)$ admit the same residue conditions, which appear in the left side of \eqref{6.3} denoted by a nilpotent matrix $\mathcal {Q}_{k}$. Taking the Laurent expansions yields
\begin{align*}
M^{(2)}(z)=C(z_{k})\left[\frac{\mathcal {Q}_{k}}{z-z_{k}}+\mathbb{I}\right]+\mathcal {O}(z-z_{k}),\\
M_{RHP}(z)=\widehat{C}(z_{k})\left[\frac{\mathcal {Q}_{k}}{z-z_{k}}+\mathbb{I}\right]+\mathcal {O}(z-z_{k}),
\end{align*}
where $C(z_{k})$ and $\hat{C}(z_{k})$ are constant terms,  we then can derive that
\begin{align*}
M^{(2)}(z)M_{RHP}^{(2)}(z)^{-1}=\mathcal {O}(1),
\end{align*}
which means has removable singularities at $z_{k}$. From the definition of $M^{(3)}(z)$, it can be obtained by direct calculation
\begin{align*}
\bar{\partial}M^{(3)}(z)&=\bar{\partial}(M^{(2)}(z)M_{RHP}^{(2)}(z)^{-1})
=\bar{\partial}M^{(2)}(z)M_{RHP}^{(2)}(z)^{-1}
=M^{(2)}(z)\bar{\partial}R^{(2)}(z)M_{RHP}^{(2)}(z)^{-1}\\
&=M^{(2)}(z)M_{RHP}^{(2)}(z)^{-1}(M_{RHP}^{(2)}(z)\bar{\partial}R^{(2)}(z)
M_{RHP}^{(2)}(z)^{-1})=M^{(3)}(z)W^{(3)}(z).
\end{align*}
\end{proof}

\begin{rem}
We transform jump contour $\Sigma^{(2)}$ into contour $\Sigma^{(3)}$ shown in Fig. 6. By comparing the contours $\Sigma^{(2)}$  and $\Sigma^{(3)}$, we find that a new curve is introduced on contour $\Sigma^{(3)}$ with jump matrix defined by
\begin{align*}
v_{9}=\left\{\begin{aligned}
&\mathbb{I},\qquad\qquad\qquad\qquad \qquad~~ z\in\left(-i\frac{z_{1}-z_{0}}{2}\tan(\pi/12),i\frac{z_{1}-z_{0}}{2}\tan(\pi/12)\right),\\
&\left(
  \begin{array}{cc}
    1 & (R_{7^{-}}-R_{7^{+}})e^{-2i\theta} \\
    0 & 1 \\
  \end{array}
\right),\quad z\in\left(i\frac{z_{1}-z_{0}}{2}\tan(\pi/12),iz_{1}\right),\\
&\left(
  \begin{array}{cc}
    1 & 0 \\
    (R_{8^{-}}-R_{8^{+}})e^{2i\theta} & 1 \\
  \end{array}
\right),\quad ~~ z\in\left(-iz_{1},-i\frac{z_{1}-z_{0}}{2}\tan(\pi/12)\right).
\end{aligned}\right.
\end{align*}
Similar to \cite{JQ-19}, we can derive that
\begin{align*}
|v_{9}-\mathbb{I}|\lesssim e^{-t},
\end{align*}
which implies that the contour does not contribute to the long-time asymptotic solution as $t\rightarrow\infty$.

\centerline{\begin{tikzpicture}[scale=0.6]
\draw[->][thick](2,0)--(4,-2);
\draw[->][thick](2,0)--(4,2);
\draw[-][thick](0,2)--(5,-3)[thick]node[right]{$\Sigma_{4}^{(3)}$};
\draw[->][thick](0,2)--(1,1);
\draw[->][thick](0,-2)--(1,-1);
\draw[->][thick](-5,-3)--(-4,-2);
\draw[->][thick](-5,3)--(-4,2);
\draw[->][thick](-2,0)--(-1,1);
\draw[->][thick](-2,0)--(-1,-1);
\draw[-][thick](0,-2)--(5,3)[thick]node[right]{$\Sigma_{1}^{(3)}$};
\draw[-][thick](0,2)--(-5,-3)[thick]node[left]{$\Sigma_{3}^{(3)}$};
\draw[-][thick](0,-2)--(-5,3)[thick]node[left]{$\Sigma_{2}^{(3)}$};
\draw[fill] (-1.2,1.2) [thick]node[left]{$\Sigma_{5}^{(3)}$};
\draw[fill] (-1.2,-1.2) [thick]node[left]{$\Sigma_{7}^{(3)}$};
\draw[fill] (1.3,1.2) [thick]node[right]{$\Sigma_{6}^{(3)}$};
\draw[fill] (1.3,-1.2) [thick]node[right]{$\Sigma_{8}^{(3)}$};
\draw[fill] (-2,0) [thick]node[below]{$z_{0}$};
\draw[fill] (2,0) [thick]node[below]{$z_{1}$};
\draw[fill] (0,0) [thick]node[right]{$\Sigma_{9}^{(3)}$};
\draw[-][fill](0,2)--(0,1);
\draw[-][fill](0,1)--(0,0);
\draw[<-][fill](0,0)--(0,-1);
\draw[-][fill](0,-1)--(0,-2);
\end{tikzpicture}}
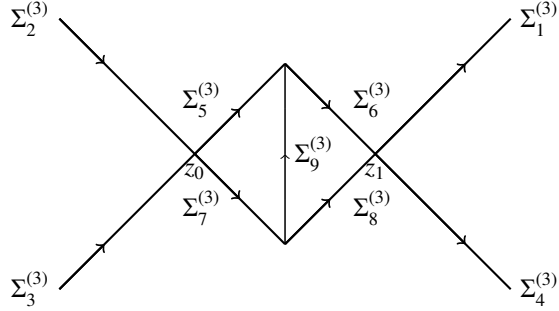
\centerline{\noindent {\small \textbf{Figure 6.} The contour of  $\Sigma^{(3)}$}.}
\end{rem}

To construct the solution of \textbf{RH Problem} \ref{RHP4}, we define
\begin{align}\label{6.7}
M_{RHP}^{(2)}(z)=\left\{\begin{aligned}
&E(z)M^{(out)}(z),\quad z\notin\left\{ \mathcal {U}_{0}\cup \mathcal {U}_{1}\right\},\\
&E(z)M^{(z_{0})}(z),\quad z\in\mathcal {U}_{0},\\
&E(z)M^{(z_{1})}(z),\quad z\in\mathcal {U}_{1},
\end{aligned}\right.
\end{align}
where $\mathcal {U}_{0}$ and $\mathcal {U}_{1}$ denote the neighborhoods of $z_{0}$ and $z_{1}$, respectively,
\begin{align*}
\mathcal {U}_{0}=\left\{z:|z-z_{0}|\leq \text{min}\left\{\frac{z_{0}}{2},\rho/3\right\}\triangleq\varepsilon\right\},\\
\mathcal {U}_{1}=\left\{z:|z-z_{1}|\leq \text{min}\left\{\frac{z_{1}}{2},\rho/3\right\}\triangleq\varepsilon\right\},
\end{align*}
which imply that $M_{RHP}^{(2)}(z)$, $M^{(z_{0})}(z)$ and $M^{(z_{1})}(z)$ have no poles in
$\mathcal {U}_{0}$ and $\mathcal {U}_{1}$ from $\text{dist}(\mathcal {Z}\cup\overline{\mathcal {Z}},R)>\rho$.  The matrix $M_{RHP}^{(2)}(z)$ is divided into two parts by the decomposition: one can be called the external model RH problem denoted by $M^{(out)}(z)$, which can be solved directly by considering the standard RH problem under the condition of reflection-less potential. The other is in the neighborhood of the phase points $M^{(z_{0})}(z)$ and $M^{(z_{1})}(z)$, which can be matched to the known model,  namely the parabolic cylinder model in $\mathcal {U}_{0}$ and $\mathcal {U}_{1}$, to solve in section 8. In addition, the matrix $E(z)$ is a error function, which can be solved by a small-norm RH problem in section 9.

\begin{prop}\label{prop-V2}
The jump matrices defined by \textbf{RH Problem} \ref{RHP4} satisfy the following estimates:\\
For fixed $\epsilon>0$, we define
\begin{align*}
\mathcal {L}_{\epsilon}&=\left\{z:z=z_{1}+uz_{1}e^{3i\pi/4},\epsilon\leq u\leq\sqrt{2}\right\},\\
&\cup\left\{z:z=z_{1}+uz_{1}e^{i\pi/4},\epsilon\leq u\leq\infty\right\},\\
&\cup\left\{z:z=z_{0}+uz_{0}e^{i\pi/4},\epsilon\leq u\leq\sqrt{2}\right\},\\
&\cup\left\{z:z=z_{0}+uz_{0}e^{3i\pi/4},\epsilon\leq u\leq\infty\right\},
\end{align*}
then the estimates can be obtained
\begin{align}
&\left|\left|V^{(2)}-\mathbb{I}\right|\right|_{L^{\infty}(\Sigma_{+}^{(2)}\cap\mathcal {U}_{1})}=\mathcal {O}\left(t^{-5/6}\left|z-z_{1}\right|^{-5/6}\right),\label{6.8}\\
&\left|\left|V^{(2)}-\mathbb{I}\right|\right|_{L^{\infty}(\Sigma_{-}^{(2)}\cap\mathcal {U}_{0})}=\mathcal {O}\left(t^{-5/6}\left|z-z_{0}\right|^{-5/6}\right),\label{6.9}\\
&\left|\left|V^{(2)}-\mathbb{I}\right|\right|_{L^{\infty}(\Sigma^{(2)}\setminus(\mathcal {U}_{1}\cup\mathcal {U}_{0}))}=\mathcal {O}\left(e^{-2t\varepsilon}\right), \label{6.9-1}
\end{align}
where the contours are defined by
\begin{align*}
\Sigma_{+}^{(2)}=\Sigma_{1}\cup\Sigma_{6}\cup\Sigma_{8}\cup\Sigma_{4},\quad
\Sigma_{-}^{(2)}=\Sigma_{5}\cup\Sigma_{2}\cup\Sigma_{3}\cup\Sigma_{7}.
\end{align*}
\end{prop}
The proposition implies that the jump matrix $V^{(2)}(z)$ uniformly goes to $I$ on both $\Sigma^{(2)}\setminus(\mathcal {U}_{1}\cup\mathcal {U}_{0})$, in addition, outside the $\mathcal {U}_{0}\cup\mathcal {U}_{1}$ there is only exponentially small error (in $t$) by completely ignoring the jump condition of $M_{RHP}^{(2)}(z)$.

\section{Outer model RH problem}
In this section, an external RH problem will be established and its solution can be approximated by a finite number of soliton solutions.
\subsection{The existence of soliton solution}
\begin{RHP}\label{RH-out}
The matrix value function $M^{(out)}(x,t;z)$, satisfing\\
$(I)$ $M^{(out)}(x,t;z)$ is analytical in $\mathbb{C}\setminus(\Sigma^{(2)}\cup\mathcal{Z}\cup\mathcal{Z}^{*})$.\\
$(II)$ As $z\rightarrow\infty$,
       \begin{align}\label{7.1}
       M^{(out)}(x,t;z)=\mathbb{I}+\mathcal {O}(z^{-1}).
       \end{align}
$(III)$ $M^{(out)}(x,t;z)$ has simple poles at each point in $\mathcal{Z}\cup\mathcal{Z}^{*}$ admitting the same residue condition in \textbf{RH Problem} \ref{RHP3} with $M^{(out)}(x,t;z)$ replacing $M^{(2)}(x,t;z)$.
\end{RHP}

Before proving the existence and uniqueness of the solution of the \textbf{RH Problem} \ref{RH-out}, we first consider the case of the RH problem without reflection, where \textbf{RH Problem} \ref{RHP1} is reduced to the following RH problem
\begin{RHP}\label{RHP-TH}
For the discrete date $\sigma_{d}=\left\{(z_{k},c_{k})\right\}_{k=1}^{N}$, and the set $\mathcal {Z}=\left\{z_{k}\right\}_{k=1}^{N}$, a new matrix valued function $m(x,t;z|\sigma_{d})$ such that:\\
$(I)$ $m(x,t;z|\sigma_{d})$ is analytic in $\mathbb{C}\setminus\left(\Sigma^{(2)}\cup\mathcal{Z}\cup\mathcal{Z}^{*}\right)$.\\
$(II)$ The asymptotic behaviour:
\begin{align}\label{7.2}
m(x,t;z|\sigma_{d})\rightarrow \mathbb{I}+\mathcal {O}(z^{-1}),\quad as~ z\rightarrow\infty.
\end{align}\\
$(III)$ Symmetry:
\begin{align}\label{7.3}
\overline{m(x,t;\overline{z}|\sigma_{d})}=\sigma_{2}m(x,t;z|\sigma_{d})\sigma_{2}.
\end{align}
$(III)$ The residue conditions: $m(x,t;z|\sigma_{d})$ has simple poles at each point in
$\mathcal{Z}\cup\mathcal{Z}^{*}$ satisfying
\begin{align}\label{7.4}
\begin{aligned}
&\mathop{Res}_{z=z_{k}}m(x,t;z|\sigma_{d})=\mathop{lim}_{z\rightarrow z_{k}}m(x,t;z|\sigma_{d})\mathcal {Q}_{k},\\
&\mathop{Res}_{z=z_{k}^{*}}m(x,t;z|\sigma_{d})=\mathop{lim}_{z\rightarrow z_{k}^{*}}m(x,t;z|\sigma_{d})\sigma_{2}\mathcal {Q}^{*}_{k}\sigma_{2},
\end{aligned}
\end{align}
where $\mathcal {Q}_{k}$ is a nilpotent matrix
\begin{align}\label{7.5}
\begin{split}
&\mathcal {Q}_{k}=\left(\begin{aligned}
\begin{array}{cc}
  0 & 0 \\
  \gamma_{k}(x,t) & 0
\end{array}
\end{aligned}\right),~~
\gamma_{k}(x,t)=c_{k}e^{2it\theta(z_{k})},\\
&\theta(z)=z\frac{x}{t}+2\alpha z^{2}+4\beta z^{3}.
\end{split}
\end{align}
\end{RHP}

The uniqueness of \textbf{RH Problem} \ref{RHP-TH} is the direct result of Liouville theorem.
From the symmetry \eqref{7.3}, taking the following expansion
\begin{align}\label{7.6}
m(x,t;z|\sigma_{d})=\mathbb{I}+\sum_{k=1}^{N}\left[\frac{1}{z-z_{k}}\left(\begin{aligned}
\begin{array}{cc}
  \zeta_{k}(x,t) & 0 \\
  \eta_{k}(x,t) & 0
\end{array}
\end{aligned}\right)+\frac{1}{z-z^{*}_{k}}\left(\begin{aligned}
\begin{array}{cc}
  0 & -\eta^{*}_{k}(x,t) \\
  0 & \zeta^{*}_{k}(x,t)
\end{array}
\end{aligned}\right)\right],
\end{align}
where $\zeta_{k}(x,t)$ and $\eta_{k}(x,t)$ are unknown coefficients to be determined. Similar to \cite{AIHP}  we can prove the existence of the solution for the
\textbf{RH Problem} \ref{RHP-TH}.

Note that the trace formula can be written as under the reflection-less case
\begin{align}\label{7.7}
s_{11}(z)=\prod_{k=1}^{N}\frac{z-z_{k}}{z-z_{k}^{*}}.
\end{align}
Let $\triangle\subseteqq \{1,2,\cdots,N\}$ and define
\begin{align}\label{7.8}
s_{11,\triangle}(z)=\prod_{k\in\triangle}\frac{z-z_{k}}{z-z_{k}^{*}}.
\end{align}
Taking the following transformation, we obtain a new matrix valued function $m^{\triangle}(z|D)$ in which the poles in the column are divided according to the choice of $\triangle$
\begin{align} \label{7.9}
m^{\triangle}(z|D)=m(z|\sigma_{d})(z)s_{11,\triangle}(z)^{\sigma_{3}},
\end{align}
where the scattering data
\begin{align}\label{7.10}
D=\{(z_{k},c'_{k})\}_{k=1}^{N},\quad c'_{k}=c_{k}s_{11,\triangle}(z)^{2}.
\end{align}
\begin{RHP}\label{RHP-DS}
For the scattering data defined by \eqref{7.10}, the matrix $m^{\triangle}(z|D)$ satisfies\\
$(I)$ Analyticity: $m^{\vartriangle}(x,t;z|D)$ is analytic  in $\mathbb{C}\setminus(\mathcal{Z}\bigcup\mathcal{Z}^{*})$.\\
$(II)$ Symmetry:
\begin{align}\label{7.11}
\overline{m^{\vartriangle}(x,t;z|D)}=\sigma_{2}m(x,t;z|\sigma_{d})\sigma_{2}.
\end{align}
$(III)$ Asymptotic behavior:
$m^{\vartriangle}(x,t;z|D)=\mathbb{I}+\mathcal {O}(z^{-1})$, \quad $z\rightarrow\infty$.\\
$(IV)$ The residue conditions: $m^{\vartriangle}(x,t;z|D)$ has simple poles at the set $\mathcal{Z}\bigcup\mathcal{Z}^{*}$
\begin{align}\label{7.12}
\begin{aligned}
&\mathop{Res}_{z=z_{k}}m^{\triangle}(x,t;z|D)=\mathop{lim}_{z\rightarrow z_{k}}m(x,t;z|\sigma_{d})\mathcal {Q}_{k}^{\vartriangle},\\
&\mathop{Res}_{z=z_{k}^{*}}m^{\triangle}(x,t;z|D)=\mathop{lim}_{z\rightarrow z_{k}^{*}}m(x,t;z|\sigma_{d})\sigma_{2}\mathcal {Q}^{\vartriangle*}_{k}\sigma_{2},
\end{aligned}
\end{align}
where $\mathcal {Q}_{k}^{\vartriangle}$ is a nilpotent matrix
\begin{align}\label{7.13}
&\mathcal {Q}_{k}^{\vartriangle}=\left\{
                                   \begin{aligned}
\left(
  \begin{array}{cc}
    0 & \gamma_{k}^{\vartriangle} \\
    0 & 0 \\
  \end{array}
\right),\quad k\in \vartriangle,\\
\left(
  \begin{array}{cc}
    0 & 0 \\
    \gamma_{k}^{\vartriangle} & 0 \\
  \end{array}
\right),\quad k\notin \vartriangle,
\end{aligned}\right.~~\gamma_{k}^{\vartriangle}=\left\{
                                   \begin{aligned}
&c_{k}^{-1}(s_{11,\vartriangle}^{'}(z_{k}))^{-2}e^{-2it\theta(z_{k})},\quad k\in \vartriangle,\\
&c_{k}(s_{11,\vartriangle}(z_{k}))^{2}e^{2it\theta(z_{k})},\qquad k\notin \vartriangle,
\end{aligned}\right.\\
&\theta(z)=z\frac{x}{t}+2\alpha z^{2}+4\beta z^{3}.\notag
\end{align}

\end{RHP}
\begin{rem}
It is noticed that the solution of \textbf{RH Problem} \ref{RHP-TH} is existent and unique, and it follows from the transformation \eqref{7.9} that the solution of \textbf{RH Problem} \ref{RHP-DS}  is also existent and unique.
\end{rem}

In \textbf{RH Problem} \ref{RHP-DS}, taking $\triangle=\triangle_{z_{1}}^{-}$ and  replacing the scattering data $D$ with
\begin{align}\label{7.14}
\widetilde{D}=\left\{(z_{k},\widetilde{c}_{k})\right\}_{k=1}^{N}, \quad
\widetilde{c}_{k}=c_{k}\delta(z_{k})^{2},
\end{align}
we then obtain
\begin{cor}\label{cor1}
There exists a unique solution to the \textbf{RH Problem} \ref{RHP-DS} satisfying
\begin{align}\label{7.15}
M^{(out)}(z)(x,t;z)=m^{\triangle_{z_{1}}^{-}}(x,t;z|\widetilde{D}),
\end{align}
where the scattering data $\widetilde{D}$ is defined by \eqref{7.6}.
\end{cor}
\begin{cor}\label{cor2}
The soliton solution of the Hirota equation \eqref{1.1} is determined by
\begin{align}
q_{sol}(x,t)&=2i\lim_{z\rightarrow\infty}(zM(z))_{12}=2i\lim_{z\rightarrow\infty}
(zm(x,t;z|\sigma_{d}))_{12}\notag \\&=2i\lim_{z\rightarrow\infty}\left(zm^{\triangle_{z_{1}}^{-}}
(z|\sigma_{d}^{\triangle_{z_{1}}^{-}})s_{11,\triangle_{z_{1}}^{-}}(z)^{-\sigma_{3}}\right)
_{12}\notag \\
&=2i\lim_{z\rightarrow\infty}\left((zm^{\triangle_{z_{1}}^{-}}
z|\sigma_{d}^{\triangle_{z_{1}}^{-}})\right)_{12}
=2i\lim_{z\rightarrow\infty}(zM^{(out)}(z))_{12}.\label{7.16}
\end{align}
\end{cor}
\subsection{The long-time behavior of soliton solution}
Taking $N=1$, $\sigma_{d}=\left\{z_{1}=\xi+i\eta, c_{1}\right\}$, and employing the residue conditions \eqref{7.12} we have the $1$-soliton solution
\begin{align}\label{7.17}
q_{sol}(x,t)=2\eta\text{sech}\left(2\eta(x+4\alpha\xi t+4\beta(3\xi\eta^{2}-\eta^{2})t+x_{0})\right)h(x,t).
\end{align}
where $x_{0}=-\frac{1}{2\eta}\log\left(\frac{|c_{1}|}{2\eta}\right)$, and $h(x,t)=e^{-2i\left[x\xi+2\alpha(\eta^{2}-\eta^{2})t+4\beta(\xi^{3}-3\xi\eta^{2})
\right]-i(\frac{\pi}{2}+arg(c_{1}))}$, from which the speed of the soliton solution is
\begin{align}\label{7.18}
v_{1}=-4\left[\alpha Rez_{1}+\beta Imz_{1}^{2}(3Rez_{1}-1)\right].
\end{align}

For given $x_{1}\leqslant x_{2}$, and velocities $v_{1}\leqslant v_{2}$ with $x_{1},x_{2},v_{1},v_{2}\in \mathbb{R}$, we define a cone and obtain the discrete spectral distribution shown in Fig. 7 and Fig. 8.
\begin{align}\label{7.19}
C(x_{1},x_{2},v_{1},v_{2})=\left\{(x,t),x=x_{0}+vt ~\text{with} ~x_{0}\in[x_{1},x_{2}],v\in[v_{1},v_{2}]\right\}.
\end{align}
\centerline{\begin{tikzpicture}[scale=0.8]
\path [fill=pink] (-1,3)--(0,0) to (2,0) -- (3,3);
\path [fill=pink] (-1,-3)--(0,0) to (2,0) -- (3,-3);
\draw[-][thick](-4,0)--(-3,0);
\draw[-][thick](-3,0)--(-2,0);
\draw[-][thick](-2,0)--(-1,0);
\draw[-][thick](-1,0)--(0,0);
\draw[-][thick](0,0)--(1,0);
\draw[-][thick](1,0)--(2,0);
\draw[-][thick](2,0)--(3,0);
\draw[-][thick](3,0)--(4,0);
\draw[->][thick](4,0)--(5,0)[thick]node[right]{$x$};
\draw[<-][thick](-2,3)[thick]node[right]{$t$}--(-2,2);
\draw[-][thick](-2,2)--(-2,1);
\draw[-][thick](-2,1)--(-2,0);
\draw[-][thick](-2,0)--(-2,-1);
\draw[-][thick](-2,-1)--(-2,-2);
\draw[-][thick](-2,-2)--(-2,-3);
\draw[fill] (0,0) circle [radius=0.08];
\draw[fill] (2,0) circle [radius=0.08];
\draw[fill] (-0.5,0)node[below]{$x_{2}$};
\draw[fill] (2.5,0)node[below]{$x_{1}$};
\draw[fill] (3.5,3)node[above]{$x=v_{2}t+x_{2}$};
\draw[fill] (3,-3)node[below]{$x=v_{1}t+x_{2}$};
\draw[fill] (-1,-3)node[below]{$x=v_{2}t+x_{1}$};
\draw[fill] (-2,3)node[above]{$x=v_{2}t+x_{1}$};
\draw[fill] (1,2)node[below]{$S$};
\draw[-][thick](-1,3)--(0,0);
\draw[-][thick](3,3)--(2,0);
\draw[-][thick](-1,-3)--(0,0);
\draw[-][thick](3,-3)--(2,0);
\end{tikzpicture}}
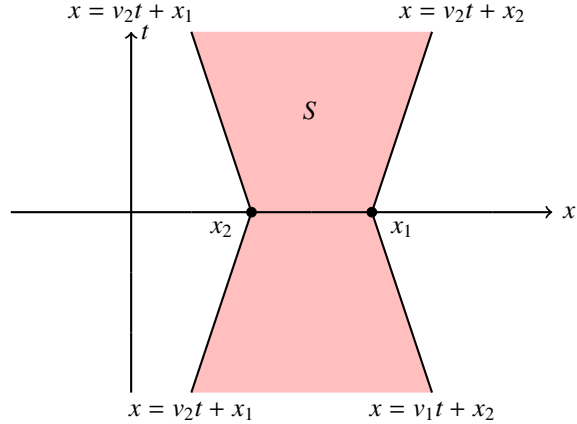
\centerline{\noindent {\small \textbf{Figure 7.} Space-time $C(x_{1},x_{2},v_{1},v_{2})$.}}

\centerline{\begin{tikzpicture}[scale=0.8]
\path [fill=pink] (1.5,3)--(-1.5,3) to (-1.5,-3) -- (1.5,-3);
\draw[-][thick](-4,0)--(-3,0);
\draw[-][thick](-3,0)--(-2,0);
\draw[-][thick](-2,0)--(-1,0);
\draw[-][thick](-1,0)--(0,0);
\draw[-][thick](0,0)--(1,0);
\draw[-][thick](1,0)--(2,0);
\draw[-][thick](2,0)--(3,0);
\draw[->][thick](3,0)--(4,0)[thick]node[right]{$Rez$};
\draw[-][thick](-1.5,3)--(-1.5,2);
\draw[-][thick](-1.5,2)--(-1.5,1);
\draw[-][thick](-1.5,1)--(-1.5,0);
\draw[-][thick](-1.5,0)--(-1.5,-1);
\draw[-][thick](-1.5,-1)--(-1.5,-2);
\draw[-][thick](-1.5,-2)--(-1.5,-3);
\draw[-][thick](1.5,3)--(1.5,2);
\draw[-][thick](1.5,2)--(1.5,1);
\draw[-][thick](1.5,1)--(1.5,0);
\draw[-][thick](1.5,0)--(1.5,-1);
\draw[-][thick](1.5,-1)--(1.5,-2);
\draw[-][thick](1.5,-2)--(1.5,-3);
\draw[fill] (1.5,0) circle [radius=0.08];
\draw[fill] (-1.5,0) circle [radius=0.08];
\draw[fill] (2,0)node[below]{$f(v_{1})$};
\draw[fill] (-2,0)node[below]{$f(v_{2})$};
\draw[fill] (3,1)node[below]{$z_{5}$} circle [radius=0.08];
\draw[fill] (3,-1)node[below]{$z^{*}_{5}$} circle [radius=0.08];
\draw[fill] (-1,-1)node[below]{$z^{*}_{6}$} circle [radius=0.08];
\draw[fill] (0.5,0.5)node[below]{$z_{7}$} circle [radius=0.08];
\draw[fill] (0.5,-0.5)node[below]{$z^{*}_{7}$} circle [radius=0.08];
\draw[fill] (-1,1)node[below]{$z_{6}$} circle [radius=0.08];
\draw[fill] (2,3)node[below]{$z_{2}$} circle [radius=0.08];
\draw[fill] (2,-3)node[below]{$z^{*}_{2}$} circle [radius=0.08];
\draw[fill] (0.5,2.5)node[below]{$z_{1}$} circle [radius=0.08];
\draw[fill] (0.5,-2.5)node[below]{$z^{*}_{1}$} circle [radius=0.08];
\draw[fill] (-0.5,2.8)node[below]{$z_{3}$} circle [radius=0.08];
\draw[fill] (-0.5,-2.8)node[below]{$z^{*}_{3}$} circle [radius=0.08];
\draw[fill] (-3,1.5)node[below]{$z_{4}$} circle [radius=0.08];
\draw[fill] (-3,-1.5)node[below]{$z^{*}_{4}$} circle [radius=0.08];
\end{tikzpicture}}
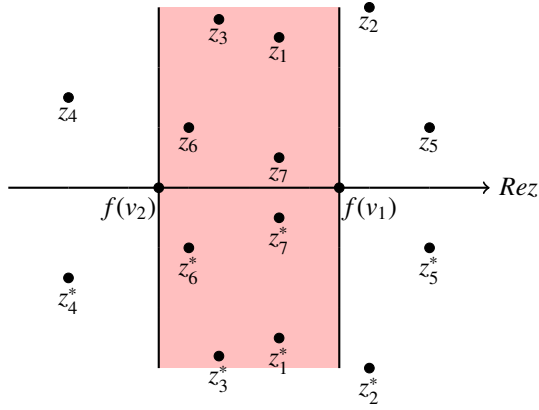
\centerline{\noindent {\small \textbf{Figure 8.} For fixed $v_{1}<v_{2}$, $I=\left[f(v_{2}),f(v_{1})\right]$.}}

In additional, we denote
\begin{align}
&\mathcal {I}=\left\{z:f(v_{2})<|z|<f(v_{1})\right\},\quad f(v_{m})=-\left(\frac{v_{m}}{4\alpha}+\frac{\beta}{\alpha}Im z_{m}^{2}(3Rez_{m}-1)\right),\notag\\
&\mathcal {K}(\mathcal {I})=\{z_{j}\in\mathcal {K}:z_{j}\in\mathcal {I}\},\quad
N(\mathcal {I})=|\mathcal {K}(\mathcal {I})|,\notag\\
&\mathcal {K}_{+}(\mathcal {I})=\{z_{j}\in\mathcal {K}:|z_{j}|>f(v_{1})\},\notag\\
&\mathcal {K}_{-}(\mathcal {I})=\{z_{j}\in\mathcal {K}:|z_{j}|<f(v_{2})\},\notag\\
&c_{j}^{\pm}(\mathcal {I})=c_{j}\prod_{Rez_{n}\in I_{\pm}\setminus\mathcal {I}}\left(
\frac{z_{j}-z_{n}}{z_{j}-z_{n}^{*}}\right)^{2}\exp\left[\pm\frac{1}{\pi i}\int_{I_{\pm}}\frac{
\log(1+|\gamma(\zeta)|^{2})}{\zeta-z_{j}}d\zeta\right].\label{E1}
\end{align}

\begin{prop}\label{prop5}
The choice of $\triangle=\triangle_{z_{1}}^{\mp}$ in \textbf{RH Problem} \ref{RHP-DS}  ensures the following estimation with $(x,t)\in C(x_{1},x_{2},v_{1},v_{2})$ as $t\rightarrow\infty$
\begin{align}\label{7.20}
\big|\big|\mathcal {Q}_{k}^{\triangle_{z_{1}}^{\mp}}\big|\big|=
\left\{\begin{aligned}
&\mathcal {O}(1),~~\qquad z_{k}\in\mathcal {K}(\mathcal {I}),\\
&\mathcal {O}(e^{-8\mu t}),\quad z_{k}\in\mathcal {K}\setminus\mathcal {K}(\mathcal {I}).
\end{aligned}\right.
\end{align}
where $\mu(\mathcal {I})=\mathop{\min}\limits_{z_{k}\in \mathcal {K}\setminus \mathcal {K}(\mathcal {I})}\{Im(z_{k})\cdot dist(Rez_{k},I)\}.
$
\end{prop}
\begin{proof}
Taking $\triangle=\triangle_{z_{1}}^{-}$ in \textbf{RH Problem} \ref{RHP-DS}, for $z_{j}\in\mathcal {K}_{-}(\mathcal {I})$ and $(x,t)\in C(x_{1},x_{2},v_{1},v_{2})$ we have
\begin{align}\label{7.21}
\big|\big|\mathcal {Q}^{\triangle_{z_{1}}^{-}}\big|\big|\lesssim |e^{-2it\theta(z_{j})}|,
\end{align}
with
\begin{align}\label{7.22}
-2it\theta(z_{j})=-2i(x_{0}z_{j}+vtz_{j}+2\alpha z_{j}^{2}+4\beta z_{j}^{3}),
\end{align}
which implies that
\begin{align}\label{7.23}
Re(-2it\theta(z_{j}))=2Imz_{j}t\left[x_{0}+vt+4\alpha Rez_{j}+4\beta Imz_{j}(3Rez_{j}-1)\right].
\end{align}
Then we can obtain
\begin{align}\label{7.24}
\big|\big|\mathcal {Q}^{\triangle_{z_{1}}^{-}}\big|\big|=\mathcal {O}(e^{-8\mu t}).
\end{align}
\end{proof}

Based on the above estimates, the following results are derived.
\begin{prop}\label{prop6}
For the reflection-less data $D=\left\{(z_{j},c_{j})\right\}_{j=1}^{N}$, $D^{\pm}(\mathcal {I})=\left\{(z_{j},c_{j}^{\pm}(\mathcal {I}))\right.$ $\left.|z_{j}\in\mathcal {K}(\mathcal {I})\right\}$, we have the following relationship with $(x,t)\in C(x_{1},x_{2},v_{1},v_{2})$ as $t\rightarrow\infty$
\begin{align}\label{7.25}
m^{\triangle_{z_{1}}^{\mp}}(x,t;z|D)=\left(\mathbb{I}+\mathcal {O}(e^{-8\mu t})\right)m^{\triangle_{z_{1}}^{\mp}}
(x,t;z|D^{\pm}(\mathcal {I})),
\end{align}
with
\begin{align}\label{7.26}
c_{j}^{\pm}(\mathcal {I})=c_{j}\prod_{Rez_{n}\in I_{\pm\setminus\mathcal {I}}}\left(
\frac{z_{j}-z_{n}}{z_{j}-z_{n}^{*}}\right)^{2}.
\end{align}
\end{prop}
\begin{proof}
For each $z_{j}\in\mathcal {K}\setminus\mathcal {K}(\mathcal {I})$,  suppose $S_{j}$ is a circle centered on $z_{j}$ and the radius is smaller than $\mu(\mathcal {I})$. $\partial S$ is the boundary of $S$. We introduce a transformation to transform the residue conditions at each $z_{j}\in\mathcal {K}\setminus\mathcal {K}(\mathcal {I})$ into a jump condition along the circle $S_{j}$.
\begin{align}\label{7.27}
\widetilde{m}^{\triangle_{z_{1}}^{-}}(z|D)=\left\{\begin{aligned}
&m^{\triangle_{z_{1}}^{-}}(z|D)\left(\mathbb{I}-\frac{\mathcal {Q}_{j}}{z-z_{j}}\right),\qquad z\in S_{j},\\
&m^{\triangle_{z_{1}}^{-}}(z|D)\left(\mathbb{I}-\frac{\sigma_{2}\mathcal {Q}_{j}\sigma_{2}}{z-z_{j}^{*}}\right),\quad z\in S_{j},\\
&m^{\triangle_{z_{1}}^{-}}(z|D), \qquad\qquad\qquad \text{elsewhere}.
\end{aligned}\right.
\end{align}
The new matrix $\widetilde{m}^{\triangle_{z_{1}}^{-}}(z|D)$ has new jump condition in each $\partial S_{j}$ denoted by $\widetilde{V}(z)$
\begin{align}\label{7.28}
\widetilde{m}_{+}^{\triangle_{z_{1}}^{-}}(z|D)=\widetilde{m}_{-}^{\triangle_{z_{1}}^{-}}(z|D)
\widetilde{V}(z),\qquad z\in\widetilde{\Sigma},
\end{align}
where
\begin{align*}
\widetilde{\Sigma}=\cup_{z_{j}\in\mathcal {K}\setminus\mathcal {K}(\mathcal {I})}(\partial S_{j}\cup\partial S_{j}^{*}).
\end{align*}
Using \eqref{7.27} and \textbf{Proposition} \ref{prop6} leads to
\begin{align}\label{7.29}
\big|\big|\widetilde{V}(z)-\mathbb{I}\big|\big|_{L^{\infty}(\widetilde{\Sigma})}=\mathcal {O}(e^{-8\mu t}).
\end{align}
It follows that from $\widetilde{m}_{+}^{\triangle_{z_{1}}^{-}}(z|D)$ and $m_{+}^{\triangle_{z_{1}}^{-}}(z|D)$ has the same poles and residue conditions
\begin{align}\label{7.30}
\epsilon(z)= \widetilde{m}^{\triangle_{z_{1}}^{-}}(z|D)m
^{\triangle_{z_{1}}^{-}(\mathcal {I})}(z|D^{\pm}(\mathcal {I}))^{-1}
\end{align}
has no poles, but the jump condition for $\epsilon(z)$ can be written as
 as $z\in\widetilde{\Sigma}$
\begin{align}\label{7.31}
\epsilon_{+}(z)=\epsilon_{-}(z)V_{\epsilon}(z),
\end{align}
where the jump matrix
\begin{align*}
V_{\epsilon}(z)=m(z|D^{\pm}(\mathcal {I}))\widetilde{V})(z)m(z|D^{\pm}(\mathcal {I}))^{-1},
\end{align*}
from which together with Eq.\eqref{7.29} yields
\begin{align}\label{7.32}
\big|\big|V_{\epsilon}(z)-\mathbb{I}\big|\big|_{L^{\infty}(\widetilde{\Sigma})}=\mathcal {O}(e^{-8\mu t}),\quad t\rightarrow\infty.
\end{align}
\begin{cor}
Suppose that $m_{sol}(x,t;D)$ and $m_{sol}(x,t;D(\mathcal {I}))$ represent the $N$-soliton solutions of the Hirota equation \eqref{1.1} corresponding to scattering data $D$ and $D(\mathcal {I})$, respectively. For given data $\mathcal {I}$, $D(\mathcal {I})$ and $C(x_{1},x_{2},v_{1},v_{2})$, the relationship can be derived
\begin{align}\label{7.33}
2i\lim\limits_{z\rightarrow\infty}z(m(x,t;z|\sigma_{d}))_{12}=
m_{sol}(x,t;D)=m_{sol}(x,t;D(\mathcal {I}))+\mathcal {O}(e^{-8\mu t}).
\end{align}
\end{cor}
Furthermore the outer model can be expressed as
\begin{cor}
The \textbf{RH Problem} \ref{RH-out} exists an unique solution $M^{(out)}(z)$ satisfying
\begin{align}\label{7.34}
M^{(out)}(z)&=m^{\triangle_{z_{1}}^{-}}(z|D^{(out)})\notag\\
&=m^{\triangle_{z_{1}}^{-}}(x,t,z|D(\mathcal {I}))\prod_{Rez_{n}\in I_{\pm}\setminus\mathcal {I}}\left(\frac{z_{j}-z_{n}}{z_{j}-z_{n}^{*}}\right)^{-\sigma_{3}}\delta^{-\sigma_{3}}+\mathcal {O}(e^{-8\mu t}),
\end{align}
where $D^{(out)}=\left\{z_{j},c_{j}(z_{1})\right\}_{j=1}^{N}$ with
\begin{align*}
 c_{j}(z_{1})=c_{j}exp\left[-\frac{1}{\pi i}\int_{I_{+}}\frac{\log(1+|\gamma(\varsigma)|^{2})}
 {\varsigma-z}d\varsigma\right].
 \end{align*}
We further obtain the reconstruction formula
\begin{align}\label{7.35}
2i\lim\limits_{z\rightarrow\infty}z(M^{(out)})_{12}=
m (x,t;D^{(out)})=m (x,t;D^{(out)}(\mathcal {I}))+\mathcal {O}(e^{-8\mu t}),
\end{align}
with
\begin{align}\label{7.36}
m_{sol}(x,t;D^{(out)})=m_{sol}(x,t;D(\mathcal {I}))+\mathcal {O}(e^{-8\mu t}),\quad
t\rightarrow\infty.
\end{align}
\end{cor}

\end{proof}
\section{A local RH problem near phase points}
From the \textbf{Proposition} \ref{prop-V2}, we know that $V^{(2)}-\mathbb{I}$ doesn't have a uniformly small jump in the neighborhood $\mathcal {U}_{z_{0}}$ and $\mathcal {U}_{z_{1}}$, therefore we establish a local model for the error function $E(z)$ with a uniformly small jump. In addition, in this section, we separate the pure RH problem from the mixed RH problem and obtain the pure $\overline{\partial}$ problem.
\begin{RHP}\label{RHP-HR}
Find a matrix-valued function $M^{(HR)}(x,t;z)$ satisfying:\\
$(I)$ Analyticity: $M^{(HR)}(x,t;z)$ is analytic in $\mathbb{C}\setminus\Sigma^{(2)}$.\\
$(II)$ Asymptotic behavior:
\begin{align}\label{8.1}
M^{(HR)}(x,t;z)=\mathbb{I}+\mathcal {O}(z^{-1}),\quad z\rightarrow\infty.
\end{align}
$(III)$ Jump condition: $M^{(HR)}(x,t;z)$ has continuous boundary values on $\Sigma^{(2)}$ and
\begin{align}\label{8.2}
M_{+}^{(HR)}(x,t;z)=M_{-}^{(HR)}(x,t;z)V^{(HR)}(z),\quad z\rightarrow\infty.
\end{align}
where the jump matrix is determined by
\begin{align*}
V^{(HR)}(z)=\left\{\begin{aligned}
&\left(
  \begin{array}{cc}
    1 & 0 \\
    \gamma(z_{1})\delta^{-2}(z_{1})(z-z_{1})^{-2i\nu(z_{1})}e^{2it\theta} & 1 \\
  \end{array}
\right), \quad z\in\Sigma_{1},\\
&\left(
  \begin{array}{cc}
    1 & \frac{\gamma^{*}(z_{1})\delta^{2}(z_{1})}{1+|\gamma(z_{1})|^{2}}(z-z_{1})
    ^{2i\nu(z_{1})}e^{-2it\theta} \\
    0 & 1 \\
  \end{array}
\right), \quad z\in\Sigma_{6},\\
&\left(
  \begin{array}{cc}
    1 & 0 \\
    \frac{\gamma(z_{1})\delta^{-2}(z_{1})}{1+|\gamma(z_{1})|^{2}}
    (z-z_{1})^{-2i\nu(z_{1})}e^{2it\theta} & 1 \\
  \end{array}
\right), \quad z\in\Sigma_{8},\\
&\left(
  \begin{array}{cc}
    1 &  \gamma^{*}(z_{1})\delta^{2}(z_{1})(z-z_{1})^{2i\nu(z_{1})}e^{-2it\theta} \\
    0 & 1 \\
  \end{array}
\right), \quad z\in\Sigma_{4},\\
&\left(
  \begin{array}{cc}
    1 & \frac{\gamma^{*}(z_{0})\delta^{2}(z_{0})}{1+|\gamma(z_{0})|^{2}}
    (z-z_{0})^{2i\nu(z_{0})}e^{-2it\theta} \\
    0 & 1 \\
  \end{array}
\right), \quad z\in\Sigma_{5},\\
&\left(
  \begin{array}{cc}
    1 & 0 \\
    \gamma(z_{0})\delta^{-2}(z_{0})(z-z_{0})^{-2i\nu(z_{0})}e^{2it\theta} & 1 \\
  \end{array}
\right), \quad z\in\Sigma_{2},\\
&\left(
  \begin{array}{cc}
    1 &  \gamma^{*}(z_{0})\delta^{2}(z_{0})(z-z_{0})^{2i\nu(z_{0})}e^{-2it\theta} \\
    0 & 1 \\
  \end{array}
\right), \quad z\in\Sigma_{3},\\
&\left(
  \begin{array}{cc}
    1 & 0 \\
    \frac{\gamma(z_{0})\delta^{-2}(z_{0})}{1+|\gamma(z_{0})|^{2}}(z-z_{0})
    ^{-2i\nu(z_{0})}e^{2it\theta} & 1 \\
  \end{array}
\right), \quad z\in\Sigma_{7}.
\end{aligned}\right.
\end{align*}

\end{RHP}
To solve the RH problem, we need to consider the parabolic cylinder (PC) model with two stationary-phase points shown in Fig. 9.

\centerline{\begin{tikzpicture}[scale=0.8]
\draw[-][thick](0.7,1.3)--(4,-2);
\draw[-][thick](0.7,-1.3)--(4,2);
\draw[-][thick](-0.7,-1.3)--(-4,2);
\draw[-][thick](-0.7,1.3)--(-4,-2);
\draw[->][thick](2,0)--(3,1);
\draw[->][thick](0.7,1.3)--(1.2,0.8);
\draw[->][thick](0.7,-1.3)--(1.2,-0.8);
\draw[->][thick](-2,0)--(-1,-1);
\draw[->][thick](-4,2)--(-2.8,0.8);
\draw[->][thick](-4,-2)--(-2.8,-0.8);
\draw[->][thick](-2,0)--(-1,1);
\draw[->][thick](2,0)--(3,-1);
\draw[-][dashed](-6,0)--(-5,0);
\draw[-][dashed](-5,0)--(-4,0);
\draw[-][dashed](-4,0)--(-3,0);
\draw[-][dashed](-3,0)--(-2,0);
\draw[-][dashed](-2,0)--(-1,0);
\draw[-][dashed](-1,0)--(0,0);
\draw[-][dashed](0,0)--(1,0);
\draw[-][dashed](1,0)--(2,0);
\draw[-][dashed](2,0)--(3,0);
\draw[-][dashed](3,0)--(4,0);
\draw[-][dashed](4,0)--(5,0);
\draw[-][dashed](5,0)--(6,0);
\draw[fill] (-1.9,-0.1) node [below]{$\Sigma_{z_{0}}$};
\draw[fill] (2.1,-0.1) node [below]{$\Sigma_{z_{1}}$};
\end{tikzpicture}}
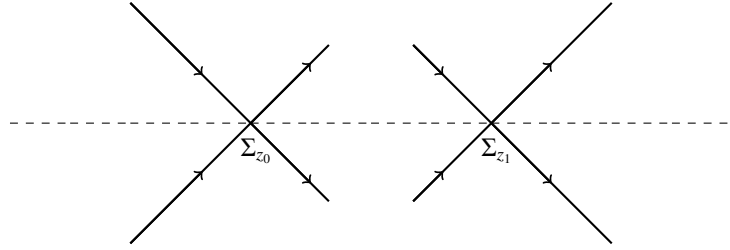
\centerline{\noindent {\small \textbf{Figure 9.}   The jump contour for the local RH problem near $z_{0}$ and $z_{1}$.}}

Taking $z_{1}$ as an example, the other stationary-phase point can be derived similarly. The jump contour of the standard parabolic cylindrical function is shown in  \textbf{Appendix A} in detail.

Expanding the function $\theta(z)$, one has
\begin{align}\label{8.3}
\theta(z)=4\beta(z-z_{1})^{3}+(12\beta z_{1}+2\alpha)(z-z_{1})^{2}-8\beta z_{1}^{3}-2\alpha z_{1}^{2}.
\end{align}
To match the standard PC-model, we first do the following scaling transformation
\begin{align}\label{8.4}
N_{A}:f(z)\rightarrow(N_{A}f)(z)=f\left(\frac{\lambda}{\sqrt{8(6\beta z_{1}+\alpha)t}}+z_{1}\right).
\end{align}
Setting
\begin{align}\label{8.5}
\gamma_{0}=\gamma(z_{1})T_{0}^{-2}(z_{1})e^{2i\nu\log\sqrt{8(\alpha+6\beta z_{1})}}e^{2it\left[4\beta(z-z_{1})^{3}-2z_{1}^{2}(\alpha+4\beta z_{1})\right]},
\end{align}
similar to \cite{AIHP}, we have
\begin{align}\label{8.6}
M_{z_{1}}^{PC}(\lambda)=\mathbb{I}+\frac{M_{1}^{PC}(z_{1})}{i\lambda}+\mathcal {O}(\lambda^{-2}),
\end{align}
where
\begin{align}\label{8.7}
M_{1}^{PC}(z_{1})=\left(
                           \begin{array}{cc}
                             0 & \beta_{12}(\gamma_{z_{1}}) \\
                             \beta_{21}(\gamma_{z_{1}}) & 0 \\
                           \end{array}
                         \right),
\end{align}
with
\begin{align*}
\beta_{12}(\gamma_{z_{1}})=\frac{\sqrt{2\pi}e^{i\pi/4}e^{-\pi\nu(z_{1})/2}}
{\gamma_{z_{1}}\Gamma(-i\nu(z_{1}))},\quad \beta_{21}(\gamma_{z_{1}})=\frac{-\sqrt{2\pi}e^{-i\pi/4}e^{-\pi\nu(z_{1})/2}}
{r_{z_{1}}^*\Gamma(i\nu(z_{1}))}=\frac{\nu}{\beta_{12}(\gamma_{z_{1}})}.
\end{align*}
\begin{rem}
For the other stationary-phase point $z_{0}$, in the same way, we take the following
scaling transformation
\begin{align}\label{8.8}
N_{B}:f(z)\rightarrow(N_{B}f)(z)=f\left(\frac{\lambda}{\sqrt{8t(6\beta z_{1}+\alpha)}}+z_{0}\right).
\end{align}
Taking
\begin{align}\label{8.9}
M_{z_{0}}^{PC}(\lambda)=\mathbb{I}+\frac{M_{1}^{PC}(z_{0})}{i\lambda}+\mathcal {O}(\lambda^{-2}),
\end{align}
we have
\begin{align}\label{8.10}
M_{1}^{PC}(z_{0})=\left(
                           \begin{array}{cc}
                             0 & \beta_{12}(\gamma_{z_{0}}) \\
                             \beta_{21}(\gamma_{z_{0}}) & 0 \\
                           \end{array}
                         \right),
\end{align}
with
\begin{align*}
\beta_{12}(\gamma_{z_{0}})=\frac{\sqrt{2\pi}e^{i\pi/4}e^{-\pi\nu(z_{0})/2}}
{\gamma_{z_{0}}\Gamma(-i\nu(z_{0}))},\quad \beta_{21}(\gamma_{z_{0}})=\frac{-\sqrt{2\pi}e^{-i\pi/4}e^{-\pi\nu(z_{0})/2}}
{\gamma^{*}_{z_{0}}\Gamma(i\nu(z_{0}))}=\frac{\nu}{\beta_{12}(\gamma_{z_{0}})}.
\end{align*}
\end{rem}

The matrix-valued $M^{(HR)}(x,t;z)$ satisfies the asymptotic behavior
\begin{align}\label{8.11}
M^{(HR)}(x,t;z)=\mathbb{I}+\frac{1}{\lambda}\left(M_{A}^{(PC)}+M_{B}^{(PC)}\right)+\mathcal {O}(\lambda^{-2}),
\end{align}
substituting \eqref{8.4} into \eqref{8.11} yields
\begin{align}\label{8.12}
M^{(HR)}(x,t;z)=\mathbb{I}+\frac{1}{\sqrt{8t(6\beta z_{1}+\alpha)}}\frac{M_{1A}^{(PC)}(z_{1})}{z-z_{1}}+
\frac{1}{\sqrt{8t(6\beta z_{0}+\alpha)}}\frac{M_{1B}^{(PC)}(z_{0})}{z-z_{0}}.
\end{align}
For the local circular domain of $z_{n}$ ($n=0,1$), there exists a constant $c$ ,which is independent of $z$  such that
\begin{align}\label{8.13}
\big|\frac{1}{z-z_{n}}\big|<c,
\end{align}
We can obtain a consistent estimate
\begin{align}\label{8.14}
\big|M^{(HT)}-\mathbb{I}\big|\lesssim \mathcal {O}(t^{-1/2}),
\end{align}
which implies that
\begin{align}\label{8.15}
||M^{(HT)}||_{\infty}\lesssim 1.
\end{align}

By using $M^{(HT)}$ to define the local model in the circles $z\in\mathcal {U}_{z_{0}}$ and $z\in\mathcal {U}_{z_{1}}$
\begin{align}\label{8.16}
M^{(z_{0},z_{1})}(z)=M^{(out)}M^{(HT)},
\end{align}
which is a bounded function in the $\mathcal {U}_{z_{0}}$ and $\mathcal {U}_{z_{1}}$, and has the same jump matrix as $M_{RHP}^{(2)}$.
\section{The small norm RH problem for the error function}
From the definition \eqref{6.7} and \eqref{8.16}, we find a error function $E(z)$ satisfying
\begin{RHP}\label{RHP-error}
The matrix-valued function $E(z)$ such that:\\
$(I)$ Analyticity: $E(z)$ is analytic in $\mathbb{C}\setminus\Sigma^{(E)}$, where
\begin{align*}
\Sigma^{(E)}=\partial\mathcal {U}_{z_{0}}\cup\partial\mathcal {U}_{z_{1}}\cup(\Sigma^{(2)}\setminus(\mathcal {U}_{z_{0}}\cup\mathcal {U}_{z_{1}}))
\end{align*}
$(II)$ Symmetry: $\overline{E(\overline{z})}=\sigma_{2}E(z)\sigma_{2}.$\\
$(III)$ Jump condition: $E(z)$ has the continuous boundary value on $\Sigma^{(E)}$
\begin{align}\label{9.1}
E_{+}(z)=E_{-}(z)V^{(E)}(z),
\end{align}
where the jump matrix $V^{(E)}(z)$
\begin{align}\label{9.2}
V^{(E)}(z)=\left\{\begin{aligned}
&M^{(out)}(z)V^{(2)}(z)M^{(out)}(z)^{-1},\quad z\in\Sigma^{(2)}\setminus(\mathcal {U}_{z_{0}}\cup\mathcal {U}_{z_{1}}),\\
&M^{(out)}(z)M^{(HT)}(z)M^{(out)}(z)^{-1},\quad z\in\partial\mathcal {U}_{z_{0}}\cup\partial\mathcal {U}_{z_{1}}.
\end{aligned}\right.
\end{align}
$(IV)$ Asymptotic behavior:
\begin{align}\label{9.3}
E(z)=\mathbb{I}+\frac{E_{1}(z)}{z}+\mathcal {O}(^{-2}),
\end{align}
\end{RHP}
Note that from the \textbf{Proposition} \ref{prop-V2}, we have the estimate as follow
\begin{align}\label{9.3-1}
\left|\left|V^{(E)}-\mathbb{I}\right|\right|_{L^{\infty}(\Sigma^{(2)}\setminus(\mathcal {U}_{1}\cup\mathcal {U}_{0}))}\lesssim \mathcal {O}\left(e^{-2t\varepsilon}\right).
\end{align}
Next we will show that the error function $E(z)$ can solve the small norm RH problem for a sufficiently large time.
\begin{prop}\label{prop7}
For $z\in\partial\mathcal {U}_{z_{0}}\cup\partial\mathcal {U}_{z_{1}}$, $M^{(out)}$ is bounded, the following estimate can be obtained
\begin{align}\label{9.4}
\big|V^{(E)}-\mathbb{I}\big|=\big|M^{(out)}(z)^{-1}(M^{(HT)}-\mathbb{I})M^{(out)}(z)\big|=\mathcal {O}(t^{-1/2}).
\end{align}
\end{prop}
The \textbf{RH Problem} \ref{RHP-error} then exists a unique solution which can be expressed by a small norm RH problem
\begin{align}\label{9.5}
E(z)=\mathbb{I}+\frac{1}{2\pi i}\int_{\Sigma^{(E)}}\frac{(\mathbb{I}+\rho(s))(V^{(E)}-\mathbb{I})}{s-z}ds,
\end{align}
where $\rho\in L^{2}(\Sigma^{(E)})$ is the unique solution of the following formula
\begin{align}\label{9.6}
 (\mathbb{I}-C_{E})\rho=C_{E}(\mathbb{I}),
\end{align}
where $C_{E}$ is a integral operator defined by
\begin{align}\label{9.7}
C_{E}f(z)=C_{-}(f(V^{(E)}-\mathbb{I})),
\end{align}
with the Cauchy projection operator on $\Sigma^{(E)}$
\begin{align*}
C_{-}(f(s))=\lim_{z\rightarrow\Sigma_{-}^{(E)}}\frac{1}{2\pi i}\int_{\Sigma^{(E)}}\frac{f(s)}{s-z}ds.
\end{align*}

Combining with \eqref{9.4}, we have $||C_{E}||<1$ for the large $t$, further $1-C_{E}$ is invertible, and $\rho$ exists and is unique. In additional
\begin{align}\label{9.8}
||\rho||_{L^{2}(\Sigma^{(E)})}\lesssim \frac{||C_{E}||}{1-||C_{E}||}\lesssim|t|^{-1/2}.
\end{align}
Then the existence and boundedness of function $E(z)$ can be obtained. To construct the solution $q(x,t)$, the asymptotic behavior of function $E(z)$ needs to be considered.
\begin{prop}
As $z\rightarrow\infty$,  the asymptotic behavior of function $E(z)$ can be expressed by
\begin{align}\label{9.9}
E(z)=\mathbb{I}+\frac{E_{1}}{z}+\mathcal {O}(z^{-2}),
\end{align}
where
\begin{align}\label{9.10}
\begin{split}
E_{1}=&\frac{1}{\sqrt{8(6\beta z_{1}+\alpha)}}\left\{M^{(out)}(z_{1})\left(
                                                                      \begin{array}{cc}
                                                                        0 & \beta_{12}(\gamma_{z_{1}}) \\
                                                                        -\beta_{21}(\gamma_{z_{1}}) & 0 \\
                                                                      \end{array}
                                                                    \right)
M^{(out)}(z_{1})^{-1}\right\}\\&+
\frac{1}{\sqrt{8(6\beta z_{0}+\alpha)}}\left\{M^{(out)}(z_{0})\left(
                                                                      \begin{array}{cc}
                                                                        0 & \beta_{12}(\gamma_{z_{0}}) \\
                                                                        -\beta_{21}(\gamma_{z_{0}}) & 0 \\
                                                                      \end{array}
                                                                    \right)
M^{(out)}(z_{0})^{-1}\right\}.
\end{split}
\end{align}
\end{prop}
\begin{proof}
Employing the relationships \eqref{9.6}, \eqref{9.8} and the estimate \eqref{9.3-1}, we have
\begin{align}\label{9.11}
\begin{split}
E_{1}&=-\frac{1}{2\pi i}\oint_{\mathcal {U}_{z_{0}}\cup\mathcal {U}_{z_{1}}}(V^{(E)}(s)-I)ds
+\mathcal {O}(t^{-1}),\\
&=\frac{1}{i\sqrt{8(6\beta z_{0}+\alpha)}}M^{(out)}(z_{0})M_{1}^{(PC)}(z_{0})M^{(out)}(z_{0})^{-1}\\&+
\frac{1}{i\sqrt{8(6\beta z_{1}+\alpha)}}M^{(out)}(z_{1})M_{1}^{(PC)}(z_{1})M^{(out)}(z_{1})^{-1}
+\mathcal {O}(t^{-1}).
\end{split}
\end{align}
Taking the notation
\begin{align}\label{9.12}
2i(E_{1})_{12}=t^{-1/2}f(x,t)+\mathcal {O}(t^{-1}),
\end{align}
where
\begin{align}\label{9.13}
\begin{split}
f(x,t)=&\frac{1}{\sqrt{2(6\beta z_{1}+\alpha)}}\left(\beta_{12}(\gamma_{z_{1}})
M_{11}^{(out)}(z_{1})^{2}+\beta_{21}(\gamma_{z_{1}})M_{12}^{(out)}(z_{1})^{2}\right)\\&+
\frac{1}{\sqrt{2(6\beta z_{0}+\alpha)}}\left(\beta_{12}(\gamma_{z_{0}})
M_{11}^{(out)}(z_{0})^{2}+\beta_{21}(\gamma_{z_{0}})M_{12}^{(out)}(z_{0})^{2}\right).
\end{split}
\end{align}
\end{proof}

\section{Pure $\bar{\partial}$-Problem}
In this section we will consider the long-time asymptotic behavior of $M^{(3)}(z)$, and the solution of the RH Problem \ref{Dbar-RHP} can be expressed by the following integral equation
\begin{align}\label{10.1}
M^{(3)}(z)=\mathbb{I}-\frac{1}{\pi}\int_{\mathbb{C}}\frac{\partial M^{(3)}(s)}{z-s}\mathrm{d}A(s)=
\mathbb{I}-\frac{1}{\pi}\int_{\mathbb{C}}\frac{M^{(3)}(s)W^{(3)}(s)}{z-s}\mathrm{d}A(s),
\end{align}
where $\mathrm{d}A(s)$ is Lebesgue measure on the contour $\mathbb{C}$. Denoting $C_{z}$ is the left Cauchy-Green integral equation, i.e.,
\begin{align}\label{10.2}
fC_{z}(z)=-\frac{1}{\pi}\int_{\mathbb{C}}\frac{f(s)W^{(3)}(s)}{z-s}\mathrm{d}A(s).
\end{align}
we transform equation \eqref{10.1} into a compact form
\begin{align}\label{10.3}
(\mathbb{I}-C_{z})M^{(3)}(z)=\mathbb{I},
\end{align}
which is equivalent to
\begin{align}\label{10.4}
M^{(3)}(z)=\mathbb{I}\cdot(\mathbb{I}-C_{z})^{-1}.
\end{align}

Regarding the existence of operator $(\mathbb{I}-C_{z})^{-1}$, we give the following proposition
\begin{prop}\label{prop8}
The norm of the Cauchy-Green integral operator $C_{z}$ decays to zero as $ t\rightarrow\infty$
\begin{align}\label{10.5}
||C_{z}||_{L^{\infty}\rightarrow L^{\infty}}\lesssim t^{-1/4},
\end{align}
which indicates that $(\mathbb{I}-C_{z})^{-1}$ exists.
\end{prop}
\begin{proof}
Taking the region $\Omega_{1}$ ($s=u+iv$, $z=x+iy$) with $|u|\geq|v|$ as an example, the other regions are similarly to discuss. For any $f\in L^{\infty}$, we have
\begin{align}\label{10.6}
||fC_{z}(z)||_{L^{\infty}}&\leq||f||_{L^{\infty}} \frac{1}{\pi}\int_{\Omega_{1}}\frac{|W^{(3)}(s)|}{|z-s|}dA(s)\notag \\
&\lesssim ||f||_{L^{\infty}} \frac{1}{\pi}\int_{\Omega_{1}}\frac{|\overline{\partial}R^{(2)}(s)|}{|z-s|}dA(s),
\end{align}
by using the fact
\begin{align}\label{10.7}
\left|W^{(3)}(s)\right|\leq \left|\left|M^{(2)}_{RHP}\right|\right|_{L^{\infty}}\left|\overline{\partial}R^{(2)}(s)\right|
\left|\left|M^{(2)}_{RHP}\right|\right|_{L^{\infty}}^{-1}\lesssim\left|\overline{\partial}R^{(2)}(s)\right|.
\end{align}
Obviously we only consider the integral equation
\begin{align}\label{10.8}
\frac{1}{\pi}\int_{\Omega_{1}}\frac{|\overline{\partial}R^{(2)}(s)|}{|z-s|}dA(s),
\end{align}
from which together  with \eqref{5.14} and
\begin{align*}
Re(2it\theta)&=8it[4\beta(3iu^{2}v-iv^{3})+2i(12\beta z_{1}+2\alpha)uv]\\
&\lesssim 8t[4\beta(-3u^{2}v+u^{2}v)-2(12\beta z_{1}+2\alpha)uv]\\
&\lesssim 8t[-8\beta u^{2}v-24\beta z_{1}uv]\\
&\lesssim -8tz_{1}uv,
\end{align*}
we have
\begin{align}\label{10.9}
\int_{\Omega_{1}}\frac{|\overline{\partial}R^{(2)}(s)|}{|z-s|}dA(s)\lesssim I_{1}+I_{2}+I_{3},
\end{align}
where
\begin{align}
&I_{1}=\int_{0}^{\infty}\int_{v}^{\infty}\frac{1}{|s-z|}|\overline{\partial}\Xi_{1}
|e^{-tz_{1}uv}dudv,\label{10.10}\\
&I_{2}=\int_{0}^{\infty}\int_{v}^{\infty}\frac{1}{|s-z|}|p'_{1}(Rez)|
e^{-tz_{1}uv}dudv,\label{10.11}\\
&I_{3}=\int_{0}^{\infty}\int_{v}^{\infty}\frac{1}{|s-z|}|s-z_{1}|^{-1/2}
e^{-tz_{1}uv}dudv.\label{10.12}
\end{align}
Through direct calculation we can get the following estimate
\begin{align}\label{10.13}
\left|\left|\frac{1}{|s-z|}\right|\right|_{L^{\infty}}=\left(\int_{v}^{\infty}\frac{1}{(u-x)^{2}
+(v-y)^{2}}du\right)^{1/2}\leq \frac{\pi}{|v-y|}.
\end{align}
It follows from the \textbf{Appendix B} that
\begin{align}\label{10.14}
|I_{1}|,~|I_{2}|,~|I_{3}|\lesssim t^{-1/4}.
\end{align}
It follows that
\begin{align}\label{10.15}
\int_{\Omega_{1}}\frac{|\overline{\partial}R^{(2)}(s)|}{|z-s|}dA(s)\lesssim t^{-1/4},
\end{align}
which implies that the inequality \eqref{10.5} holds.
\end{proof}

In what follows, we consider the asymptotic behavior of $M^{(3)}(z)$
\begin{align}\label{10.16}
M^{(3)}(z)=\mathbb{I}+\frac{M_{1}^{(3)}(x,t)}{z}+\mathcal {O}(z^{-2}),
\end{align}
where
\begin{align}\label{10.17}
M_{1}^{(3)}(x,t)=\frac{1}{\pi}\int_{C}M^{(3)}(s)W^{(3)}(s)dA(s).
\end{align} In addition, $M_{1}^{(3)}(x,t)$
 satisfies the estimate
\begin{prop}\label{prop9}
For $t\neq 0$, the matrix $M_{1}^{(3)}(x,t)$ such that
\begin{align}\label{10.18}
\left|M_{1}^{(3)}(x,t)\right|\lesssim t^{-3/4},
\end{align}
\end{prop}
\begin{proof}
Similar to the proof of \textbf{Proposition} \ref{prop8}, we only  consider the region $\Omega_{1}=\left\{(u+z_{1},v):v\geq 0,v\leq u<\infty\right\}$, the other case can be proved in the same way.
Note that  $M^{(3)}$ and $M^{(2)}_{RHP}$ are bounded, and employing formulas \eqref{10.17} and  \eqref{10.18}, we obtain
\begin{align}\label{10.19}
\left|\left|M^{(3)}_{1}(x,t)\right|\right|&\leq \frac{1}{\pi}\iint_{\Omega_{1}}\left|M^{(3)}(s)M^{(2)}_{RHP}(s)
\overline{\partial}R^{(2)}(s)M^{(2)}_{RHP}(s)^{-1}\right|dA(s),\\
&\leq I_{4}+I_{5}+I_{6},
\end{align}
where
\begin{align}
&I_{4}=\int_{0}^{\infty}\int_{v}^{\infty}|\overline{\partial}\Xi_{1}
|e^{-tz_{1}uv}dudv,\label{10.20}\\
&I_{5}=\int_{0}^{\infty}\int_{v}^{\infty}|p'_{1}(Rez)|
e^{-tz_{1}uv}dudv,\label{10.21}\\
&I_{6}=\int_{0}^{\infty}\int_{v}^{\infty}|s-z_{1}|^{-1/2}
e^{-tz_{1}uv}dudv.\label{10.22}
\end{align}
It follows from similar to the prove of $I_{i}$ ($i=1,2,3$) that
\begin{align}\label{10.23}
|I_{4}|,~|I_{5}|,~|I_{6}|\lesssim t^{-3/4},
\end{align}
which imply the proof of the proposition is completed.
\end{proof}

\section{Soliton resolution for the Hirota equation}
In this section we will construct the long-time asymptotic solution of the Hirota equation. According to the inverse transformation of \eqref{4.16}, \eqref{5.4}, \eqref{6.4} and \eqref{6.7}, one has
\begin{align}\label{11.1}
M(z)=M^{(3)}(z)E(z)M^{(out)}(z)R^{(2)}(z)^{-1}T^{\sigma_{3}}(z),~~ z\in\mathbb{C}\setminus(\mathcal{U}_{z_{0}}\cup\mathcal{U}_{z_{1}}),
\end{align}
note that $T^{\sigma_{3}}(z)$ is a diagonal matrix.

To construct the solution $q(x,t)$, we take $z\rightarrow\infty$ along the imaginary axis. The advantage of this is that $R^{(2)}(z)$ is the identity matrix. By using these relationships \eqref{4.13}, \eqref{7.36}, \eqref{9.3} and \eqref{9.13}, we have
\begin{align}\label{11.2}
M=\left(\mathbb{I}+\frac{M^{(3)}_{1}}{z}+\cdots\right)\left(\mathbb{I}+\frac{E_{1}}{z}+\cdots\right)
\left(\mathbb{I}+\frac{M_{1}^{(out)}}{z}+\cdots\right)\left(\mathbb{I}
+\frac{T_{1}^{\sigma_{3}}}{z}+\cdots\right),
\end{align}
from which,   comparing the coefficients of $z^{-1}$ to get
\begin{align}\label{11.3}
M_{1}=M_{1}^{(out)}+E_{1}+M_{1}^{(3)}+T_{1}^{\sigma_{3}}.
\end{align}

\section*{Acknowledgements}

This work was supported by  the National Natural Science Foundation of China under Grant No. 11975306, the Natural Science Foundation of Jiangsu Province under Grant No. BK20181351, the Six Talent Peaks Project in Jiangsu Province under Grant No. JY-059,  and the Fundamental Research Fund for the Central Universities under the Grant Nos. 2019ZDPY07 and 2019QNA35.

\section{Appendix A: The parabolic cylinder model problem}
Here we describe in detail the construction of parabolic cylindrical function solution \cite{PC-model,PC-model-2}, which is frequently used in the literature of long-time asymptotic solution. Define the contour $\Sigma^{(pc)}=\cup_{j=1}^{4}\Sigma_{j}^{(pc)}$ shown in Fig. 10 where
\begin{align}
\Sigma_{j}^{(pc)}=\left\{\lambda\in\mathbb{C}|\arg\lambda=\frac{2j-1}{4}\pi \right\}.\tag{A.1}
\end{align}
\centerline{\begin{tikzpicture}[scale=0.6]
\draw[-][dashed](-6,0)--(-5,0);
\draw[-][dashed](-5,0)--(-4,0);
\draw[-][dashed](-4,0)--(-3,0);
\draw[-][dashed](-3,0)--(-2,0);
\draw[-][dashed](-2,0)--(-1,0);
\draw[-][dashed](-1,0)--(0,0);
\draw[-][dashed](0,0)--(1,0);
\draw[-][dashed](1,0)--(2,0);
\draw[-][dashed](2,0)--(3,0);
\draw[-][dashed](3,0)--(4,0);
\draw[-][dashed](4,0)--(5,0);
\draw[-][dashed](5,0)--(6,0);
\draw[-][thick](-4,-4)--(4,4);
\draw[-][thick](-4,4)--(4,-4);
\draw[->][thick](2,2)--(3,3);
\draw[->][thick](-4,4)--(-3,3);
\draw[->][thick](-4,-4)--(-3,-3);
\draw[->][thick](2,-2)--(3,-3);
\draw[fill] (3.2,3)node[below]{$\Sigma_{1}^{(pc)}$};
\draw[fill] (3.2,-3)node[above]{$\Sigma_{4}^{(pc)}$};
\draw[fill] (-3.2,3)node[below]{$\Sigma_{2}^{(pc)}$};
\draw[fill] (-2,-3)node[below]{$\Sigma_{3}^{(pc)}$};
\draw[fill] (0,0)node[below]{$0$};
\draw[fill] (1,0)node[below]{$\Omega_{6}$};
\draw[fill] (1,0)node[above]{$\Omega_{1}$};
\draw[fill] (0,-1)node[below]{$\Omega_{5}$};
\draw[fill] (0,1)node[above]{$\Omega_{2}$};
\draw[fill] (-1,0)node[below]{$\Omega_{4}$};
\draw[fill] (-1,0)node[above]{$\Omega_{3}$};
\draw[fill] (7,3)node[below]{$\lambda^{i\nu\hat{\sigma}_{3}}e^{-\frac{i\lambda^{2}}{4}\hat{\sigma}_{3}}
\left(
  \begin{array}{cc}
    1 & 0 \\
    r_{0} & 1 \\
  \end{array}
\right)
$};
\draw[fill] (7,-2)node[below]{$\lambda^{i\nu\hat{\sigma}_{3}}e^{-\frac{i\lambda^{2}}{4}\hat{\sigma}_{3}}
\left(
  \begin{array}{cc}
    1 & r^{*}_{0} \\
    0 & 1 \\
  \end{array}
\right)
$};
\draw[fill] (-7,2.5)node[below]{$\lambda^{i\nu\hat{\sigma}_{3}}e^{-\frac{i\lambda^{2}}{4}\hat{\sigma}_{3}}
\left(
  \begin{array}{cc}
    1 & \frac{r^{*}_{0}}{1+|r_{0}|^{2}} \\
    0 & 1 \\
  \end{array}
\right)
$};
\draw[fill] (-7,-1)node[below]{$\lambda^{i\nu\hat{\sigma}_{3}}e^{-\frac{i\lambda^{2}}{4}\hat{\sigma}_{3}}
\left(
  \begin{array}{cc}
    1 & 0 \\
    \frac{r_{0}}{1+|r_{0}|^{2}} & 1 \\
  \end{array}
\right)
$};
\end{tikzpicture}}
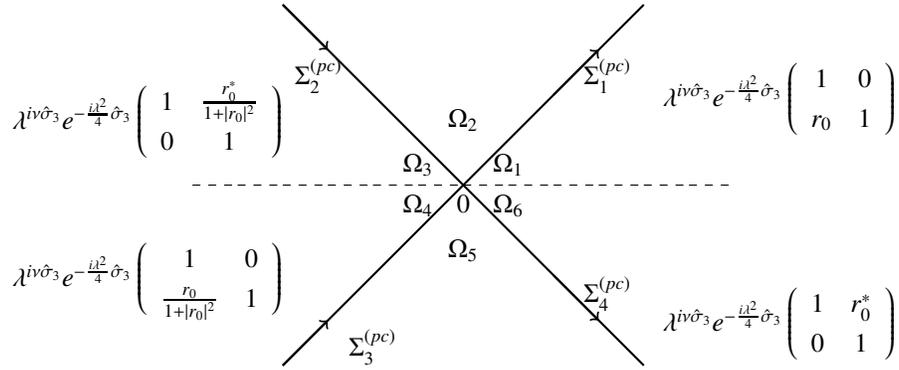
\centerline{\noindent {\small \textbf{Figure 10.} The contour of $\Sigma_{j}^{(pc)}$ which is the jump condition of the following RH problem.}}

For $r_{0}\in \mathbb{C}$, let $\nu(r)=-\frac{1}{2\pi}\log(1+|\gamma_{0}|^{2})$, we then consider the following parabolic cylinder model Riemann-Hilbert problem.
\begin{RHP}\label{PC-model}
Find a matrix-valued function $M^{(pc)}(\lambda)$ satisfying
\begin{align}
&\bullet \quad M^{(pc)}(\lambda)~ \text{is analytic in}~ \mathbb{C}\setminus\Sigma^{(pc)}, \tag{A.2}\\
&\bullet \quad M_{+}^{(pc)}(\lambda)=M_{-}^{(pc)}(\lambda)V^{(pc)}(\lambda),\quad
\lambda\in\Sigma^{pc}, \tag{A.3}\\
&\bullet \quad M^{(pc)}(\lambda)=\mathbb{I}+\frac{M_{1}}{\lambda}+O(\lambda^{2}),\quad
\lambda\rightarrow\infty. \tag{A.4}
\end{align}
where
\begin{align}\label{Vpc}
V^{(pc)}(\lambda)=\left\{\begin{aligned}
\lambda^{i\nu\hat{\sigma}_{3}}e^{-\frac{i\lambda^{2}}{4}
\hat{\sigma}_{3}}\left(
                    \begin{array}{cc}
                      1 & 0 \\
                      r_{0} & 1 \\
                    \end{array}
                  \right),\quad \lambda\in\Sigma_{1}^{(pc)},\\
\lambda^{i\nu\hat{\sigma}_{3}}e^{-\frac{i\lambda^{2}}{4}
\hat{\sigma}_{3}}\left(
                    \begin{array}{cc}
                      1 & \frac{r^{*}_{0}}{1+|r_{0}|^{2}} \\
                      0 & 1 \\
                    \end{array}
                  \right),\quad \lambda\in\Sigma_{2}^{(pc)},\\
\lambda^{i\nu\hat{\sigma}_{3}}e^{-\frac{i\lambda^{2}}{4}
\hat{\sigma}_{3}}\left(
                    \begin{array}{cc}
                      1 & 0\\
                      \frac{r_{0}}{1+|r_{0}|^{2}} & 1 \\
                    \end{array}
                  \right),\quad \lambda\in\Sigma_{3}^{(pc)},\\
\lambda^{i\nu\hat{\sigma}_{3}}e^{-\frac{i\lambda^{2}}{4}
\hat{\sigma}_{3}}\left(
                    \begin{array}{cc}
                      1 & r^{*}_{0} \\
                      0 & 1 \\
                    \end{array}
                  \right),\quad \lambda\in\Sigma_{4}^{(pc)},
\end{aligned}\right.\tag{A.5}
\end{align}
\end{RHP}

As shown in \cite{PC-equation}, the parabolic cylinder equation reads
\begin{align*}
\left(\frac{\partial^{2}}{\partial z^{2}}+(\frac{1}{2}-\frac{z^{2}}{2}+a)\right)D_{a}=0.
\end{align*}
The explicit solution $M^{(pc)}(\lambda, r_{0})$ of the \textbf{RH Problem} \ref{PC-model} can be expressed by in the literatures \cite{DZ-AM, PC-solution2}
\begin{align*}
M^{(pc)}(\lambda, r_{0})=\Phi(\lambda, r_{0})\mathcal{P}(\lambda, r_{0})e^{\frac{i}{4}\lambda^{2}\sigma_{3}}\lambda^{-i\nu\sigma_{3}},
\end{align*}
where
\begin{align*}
\mathcal{P}(\lambda, r_{0})=\left\{\begin{aligned}
&\left(
                    \begin{array}{cc}
                      1 & 0 \\
                      -r_{0} & 1 \\
                    \end{array}
                  \right),\quad &\lambda\in\Omega_{1},\\
&\left(
                    \begin{array}{cc}
                      1 & -\frac{r^{*}_{0}}{1+|r_{0}|^{2}} \\
                      0 & 1 \\
                    \end{array}
                  \right),\quad &\lambda\in\Omega_{3},\\
&\left(
                    \begin{array}{cc}
                      1 & 0\\
                      \frac{r_{0}}{1+|r_{0}|^{2}} & 1 \\
                    \end{array}
                  \right),\quad &\lambda\in\Omega_{4},\\
&\left(
                    \begin{array}{cc}
                      1 & r^{*}_{0} \\
                      0 & 1 \\
                    \end{array}
                  \right),\quad &\lambda\in\Omega_{6},\\
&~~~\mathbb{I},\quad &\lambda\in\Omega_{2}\cup\Omega_{5},
\end{aligned}\right.
\end{align*}
and
\begin{align*}
\Phi(\lambda, r_{0})=\left\{\begin{aligned}
\left(
                    \begin{array}{cc}
                      e^{-\frac{3\pi\nu}{4}}D_{i\nu}\left( e^{-\frac{3i\pi}{4}}\lambda\right) & -i\beta_{12}e^{-\frac{\pi}{4}(\nu-i)}D_{-i\nu-1}\left( e^{-\frac{i\pi}{4}}\lambda\right) \\
                      i\beta_{21}e^{-\frac{3\pi(\nu+i)}{4}}D_{i\nu-1}\left( e^{-\frac{3i\pi}{4}}\lambda\right) & e^{\frac{\pi\nu}{4}}D_{-i\nu}\left( e^{-\frac{i\pi}{4}}\lambda\right) \\
                    \end{array}
                  \right),\quad \lambda\in\mathbb{C}^{+},\\
\left(
                    \begin{array}{cc}
                      e^{\frac{\pi\nu}{4}}D_{i\nu}\left( e^{\frac{i\pi}{4}}\lambda\right) & -i\beta_{12}e^{-\frac{3\pi(\nu-i)}{4}}D_{-i\nu-1}\left( e^{\frac{3i\pi}{4}}\lambda\right) \\
                      i\beta_{21}e^{\frac{\pi}{4}(\nu+i)}D_{i\nu-1}\left( e^{\frac{i\pi}{4}}\lambda\right) & e^{-\frac{3\pi\nu}{4}}D_{-i\nu}\left( e^{\frac{3i\pi}{4}}\lambda\right) \\
                    \end{array}
                  \right),\quad \lambda\in\mathbb{C}^{-},
\end{aligned}\right.
\end{align*}
with
\begin{align*}
\beta_{12}=\frac{\sqrt{2\pi}e^{i\pi/4}e^{-\pi\nu/2}}{r_0\Gamma(-i\nu)},\quad \beta_{21}=\frac{-\sqrt{2\pi}e^{-i\pi/4}e^{-\pi\nu/2}}{r_0^*\Gamma(i\nu)}=\frac{\nu}{\beta_{12}}.
\end{align*}
It follows from the well known asymptotic behavior of $D_{a}(z)$  that the asymptotic behavior of the solution can be determined by
\begin{align}\label{A-1}
M^{(pc)}(r_0,\lambda)=\mathbb{I}+\frac{M_1^{(pc)}}{i\lambda}+O(\lambda^{-2}), \tag{A.6}
\end{align}
where
\begin{align*}
M_1^{(pc)}=\begin{pmatrix}0&\beta_{12}\\-\beta_{21}&0\end{pmatrix}.
\end{align*}

\section{Appendix B: Detailed calculations for the pure $\bar{\partial}$-Problem}
For $z\in\Omega_{1}$, let $s=u+iv$, $z=x+iy$, then we have
\begin{align}
|I_{1}|&=\int_{0}^{\infty}\int_{v}^{\infty}\frac{1}{|s-z|}|\overline{\partial}\Xi_{1}
|e^{-tz_{1}uv}dudv\notag\\
&\lesssim \int_{0}^{\infty}e^{-tz_{1}uv}\left|\left|\overline{\partial}\Xi_{1}\right|\right|_{L^{2}}
\left|\left|\frac{1}{|s-z|}\right|\right|_{L^{2}}dv\notag\\
&\lesssim\int_{0}^{\infty}\frac{1}{\sqrt{|v-\eta|}}e^{-tz_{1}v^{2}}dv\notag\\
&=\int_{0}^{\eta}\frac{1}{\sqrt{|v-\eta|}}e^{-tz_{1}v^{2}}dv+
\int_{\eta}^{\infty}\frac{1}{\sqrt{|v-\eta|}}e^{-tz_{1}v^{2}}dv\notag\\
&\triangleq F_{1}+F_{2}\tag{B.1}
\end{align}
Note that taking $w=\frac{v}{\eta}$ yields
\begin{align}\label{B.2}
F_{1}&=\int_{0}^{\eta}\frac{1}{\sqrt{\eta}{\sqrt{1-\frac{v}{\eta}}}}e^{-tz_{1}v^{2}}dv
=\int_{0}^{\eta}\frac{\sqrt{\eta}e^{-tz_{1}w^{2}\eta^{2}}}{\sqrt{1-w}}dw\notag\\
&\lesssim t^{-1/4}\int_{0}^{\eta}\frac{1}{\sqrt{w(1-w)}}dw\lesssim t^{-1/4}.\tag{B.2}
\end{align}
Similarly taking $v=w+\eta$ yields
\begin{align}\label{B.3}
F_{2}&=\int_{\eta}^{\infty}\frac{1}{\sqrt{v-\eta}}e^{-tz_{1}v^{2}}dv
=\int_{\eta}^{\infty}\frac{1}{\sqrt{w}}e^{-tz_{1}(w+\eta)^{2}}dv\notag\\
&\lesssim\int_{\eta}^{\infty}\frac{1}{\sqrt{w}}e^{-tz_{1} w^{2}}dv
\lesssim t^{-1/4}.\tag{B.3}
\end{align}
For $I_{2}$, we can obtain the estimate in the same way
\begin{align}\label{B.4}
I_{2}\lesssim t^{-1/4}.\tag{B.4}
\end{align}
It follows from the H\"{o}lder inequality (taking $q>2$) that
\begin{align}
\left|\left|\frac{1}{s-z}\right|\right|_{L^{q}}\lesssim |v-\eta|^{\frac{1}{q}-1}.\tag{B.5}
\end{align}
Then
\begin{align}\label{B.6}
I_{3}&\lesssim \int_{0}^{\infty}e^{-tz_{1}v^{2}}\left|\left|\frac{1}{\sqrt{s-z_{1}}}\right|\right|_{L^{p}}
\left|\left|\frac{1}{s-z}\right|\right|_{L^{q}}dv \notag\\
&\lesssim \int_{0}^{\infty}e^{-tz_{1}v^{2}}v^{\frac{1}{p}-\frac{1}{2}}
\left|v-\eta\right|^{\frac{1}{q}-1}dv\notag\\
&\lesssim t^{-1/4}.\tag{B.6}
\end{align}
Subsequently combining with \eqref{B.2}, \eqref{B.3}, \eqref{B.4} and \eqref{B.6}, we can obtain \eqref{10.14}.

\renewcommand{\baselinestretch}{1.2}

\end{document}